\documentclass[12pt]{amsart}

\usepackage[top=100pt,bottom=60pt,left=96pt,right=92pt]{geometry}

\usepackage[utf8]{inputenc} % in modern latex versions automatically loaded/included
\usepackage[export]{adjustbox}
\usepackage[mathscr]{eucal}
\usepackage{tikz-cd}
\usepackage{amsmath,amsthm,amssymb,amsfonts,enumerate,xcolor,layout,tikz,enumitem,verbatim} 
\usepackage[normalem]{ulem} % to strike out text with \sout{...}; attention: if [normalem] is removed then \emph{} underlines instead of making italic

\usetikzlibrary{shapes.geometric, arrows}

\usepackage[style=alphabetic,maxnames=5,sorting=nty]{biblatex}

\usepackage{chngcntr}
\counterwithin{figure}{section}

\numberwithin{equation}{section}

\usepackage{array}

\usepackage{tikz-cd} % provides diagram features of tikz
\usepackage[colorlinks]{hyperref}
\usepackage[T1]{fontenc} % necessary when characters/symbols with dots, accents etc appear
\usepackage{mathtools}

\usepackage{graphicx}

\usepackage[font=small,labelfont=bf]{caption} % let's see if we like this layout

\usepackage{subcaption} %to fix the labels on subfigures from (A) to (a) (when using amsart)
\usepackage{fancyhdr}

\definecolor{linkblue}{rgb}{0,0,.6}
\definecolor{citered}{rgb}{.7,0,0}
\hypersetup{colorlinks =true, linkcolor=linkblue, citecolor = citered, urlcolor=linkblue}

\DeclareMathOperator{\Span}{span}

\allowdisplaybreaks

\newtheorem{theorem}{Theorem}[section]
\newtheorem{proposition}[theorem]{Proposition}
\newtheorem{corollary}[theorem]{Corollary}
\newtheorem{lemma}[theorem]{Lemma}

\theoremstyle{definition}
\newtheorem{definition}[theorem]{Definition}
\newtheorem{remark}[theorem]{Remark}

\theoremstyle{plain}

\def\N{{\mathbb N}}

\def\R{{\mathbb R}}
\newcommand{\purge}[1]{} % makes everything between the curly brackets invisible

\newcommand{\vungoc}{V\~u Ng\k{o}c}

\begin{document}
\pagestyle{fancy}

\fancyhead{}
\fancyhead[CO]{(Non)displaceability of semitoric fibers}
\fancyhead[CE]{Sonja Hohloch $\&$ Pedro Santos}

\title{(Non)displaceability in semitoric systems}
\author{Sonja Hohloch \quad $\&$ \quad Pedro Santos}
\date{\today}

\begin{abstract}
We adapt and generalize McDuff's method of probes from toric systems to semitoric integrable systems and apply it to study the (non)displaceability properties of the fibers of $3$ examples of semitoric integrable systems.  
\noindent

\end{abstract}

\maketitle

\tableofcontents

\section{Introduction}

Let $(M,\omega)$ be a closed symplectic manifold, and let $\phi:M\rightarrow M$ be a Hamiltonian diffeomorphism. Determining the minimal number of fixed points of $\phi$ is a central question in Hamiltonian dynamics and lies at the heart of the Arnold conjecture. Geometrically, these fixed points are precisely the intersection points o the graph of $\phi$ and the diagonal $\Delta_{M\times M}$ in $M\times M$. Endowing $M\times M$ with the symplectic form $\omega\oplus-\omega$, both the graph of $\phi$ and the diagonal $\Delta_{M\times M}$ are Lagrangian submanifolds. Therefore, finding fixed points of a Hamiltonian diffeomorphism is equivalent to finding intersections between Lagrangian submanifolds.
This perspective motivated Floer \cite{floer1988morse,floer1988unregularized,floer1989witten} to introduce a homology theory, now known as Floer homology, in the late 1980s. Roughly speaking, under suitable assumptions, the Floer complex associated to a pair of Lagrangian submanifolds is generated by their intersection points. In particular, when the resulting Floer homology is nonzero, the Lagrangians cannot be made disjoint by a Hamiltonian diffeomorphism. Thus, Floer homology provides a powerful tool for detecting a fundamental symplectic rigidity phenomenon: the inability to displace certain subsets from each other by Hamiltonian diffeomorphisms. In this paper we study this intersection and displaceability problem for fibers of certain integrable systems using various methods.

More precisely, let $A,B\subset M$. We say that $A$ is \textbf{displaceable} from $B$ if there exists a Hamiltonian diffeomorphism $\phi$ such that $\phi(A)\cap \overline{B}=\emptyset$. Otherwise, we say that $A$ \textbf{cannot be displaced} from $B$. If $A$ cannot be displaced from itself we call it \textbf{nondisplaceable}. Otherwise we say that $A$ is \textbf{displaceable}. One can also consider displacement by arbitrary symplectomorphisms rather than only Hamiltonian diffeomorphisms; in this case we speak of \textbf{displaceability by symplectomorphisms}. These notions provide a way of measuring how much freedom a subset of a symplectic manifold has to be moved while preserving the symplectic structure.

Over the years, research on (non)displaceability developed into an area of its own, for instance, just to mention a few, it is essential for Fukaya categories (see Fukaya \& Oh \& Ohta \& Ono \cite{fukayabook1,fukayabook2}), it motivated the development of symplectic quasi-states by Entov \& Polterovich \cite{polterovich2006quasi}, and the development of explicit methods, see for example McDuff \cite{mcduff2011displacing} and Abreu \& Borman \& McDuff \cite{abreu2014displacing}, that give sufficient conditions for a Lagrangian torus fiber to be displaceable. This last aspect is in fact the starting point for the present paper, as we will explain now.

Briefly summarized, McDuff's \cite{mcduff2011displacing} explicit method (called \textit{the method of probes}) makes use of the Delzant polytope (= momentum polytope) of a symplectic toric manifold. In the world of integrable systems, toric systems are the simplest examples and are classified by their corresponding Delzant polytopes. The next natural, meaning more general, class of integrable systems are so called semitoric systems (see Definition \ref{d.semitoric}) which were classified by Pelayo \& \vungoc \ \cite{pelayo2009semitoric, pelayo2011constructing} and Palmer \& Pelayo \& Tang \cite{palmer2019semitoric} in terms of $5$ invariants. One of these invariants, the so called polytope invariant, is the generalization of the Delzant polytope of a toric system. In the present paper we study the (non)displaceability properties of the fibers of several semitoric systems.
\subsection{Our results}
Let $(M,\omega,F=(J,H))$ be a semitoric system (see Definition \ref{d.semitoric}) on a $4$-dimensional symplectic manifold. Intuitively semitoric systems are more general than toric systems by requiring only an $\mathbb{S}^1$-symmetry instead of a $\mathbb{T}^2$-symmetry and allowing for focus-focus singularities in addition to the ones appearing in the toric case. As in the toric situation the fibers are connected. A fiber with $k$ focus-focus points is often seen as a torus with $k$ pinches. Alternatively the fiber $F^{-1}(c)$ can be seen as a closed chain of $k$ Lagrangian spheres joined at the poles. If $k=1$ the sphere is immersed, and if $k>1$ the spheres are embedded. Our first result is about the (non)displaceability of focus-focus fibers with more than one focus-focus point:
\begin{theorem}
    A focus-focus fiber containing more than one focus-focus point is \textbf{nondisplaceable}.
\end{theorem}
This is proven in Theorem \ref{p.lagrangiansphere} and Corollary \ref{c.highernondisplaceable}. It is a consequence of the topological nondisplaceability of Lagrangian spheres. An explicit example of a semitoric system containing focus-focus fibers with $2$ pinches can be found in Section \ref{s.octagon}.

Therefore the question of (non)displaceability of focus-focus fibers is reduced to considering fibers containing a single focus-focus point. In Section \ref{s.displacinggeneralff} we give conditions, in terms of the invariants of the semitoric system, to displace a focus-focus fiber. 

The so called semitoric polytope invariant of $(M,\omega,F)$, for details see Section \ref{s.semitoric}, arises from $F(M)$ by introducing vertical cuts, from above or below, to the focus-focus values and then "straightening" it, using a choice of action coordinates induced by the choice of vertical cuts, to obtain a convex polytope, as illustrated in Figure \ref{f.straighteninghomeomorphism}. This yields a representative of the semitoric polytope invariant.
\begin{figure}
\begin{center}
    \includegraphics[width=0.8\textwidth]{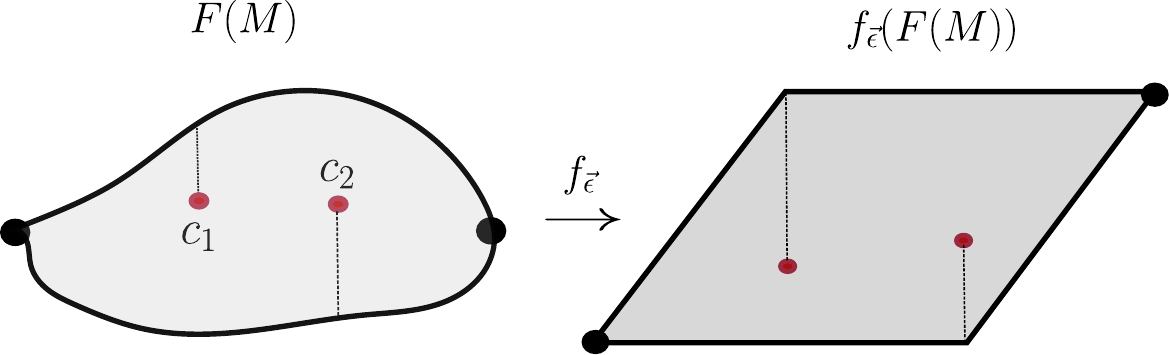}
    \caption{Straightening homeomorphism $f_{\vec{\epsilon}}$ applied to $F(M)$ to obtain a representative of the polytope invariant $f_{\vec{\epsilon}}(F(M))$. The focus-focus values of $F(M)$ are $c_1$ and $c_2$.}
    \label{f.straighteninghomeomorphism}
\end{center}
\end{figure}

In order to displace a Lagrangian torus fiber $L$ of a toric system, McDuff's method of probes embeds a disk of large enough area  times a circle, $D^2(a)\times S^1$, in an area preserving way in $(M,\omega)$. Moreover, in the method of probes, the Lagrangian torus $L$ is identified with $\partial(D^2(b))\times S^1$, where $b<\frac{a}{2}$. If (or since) $b<\frac{a}{2}$ the Lagrangian torus is displaceable. In the Delzant polytope, this corresponds to choosing a suitable ray that starts from an edge of the polytope. We apply this idea in a representative of the polytope invariant of a semitoric system to displace semitoric fibers. Compared to the toric situation, we have to deal with additional obstructions coming from the presence of focus-focus values: the cuts introduced in $F(M)$ to obtain a representative of the polytope invariant act as barriers for the ray. Furthermore, in order to displace focus-focus fibers of semitoric systems new ideas need to be introduced. The idea is to symplectically embed the product of two disks such that one disk has big enough area to be able to apply McDuff's ideas. In order to do this procedure one takes into consideration the distance to the boundary of the focus-focus value in a representative of the polytope invariant while paying attention to the cuts introduced by the other focus-focus values, see Figure \ref{f.intuitiondisplacingfffibers}.
\begin{figure}
\begin{center}
    \includegraphics[width=0.5\textwidth]{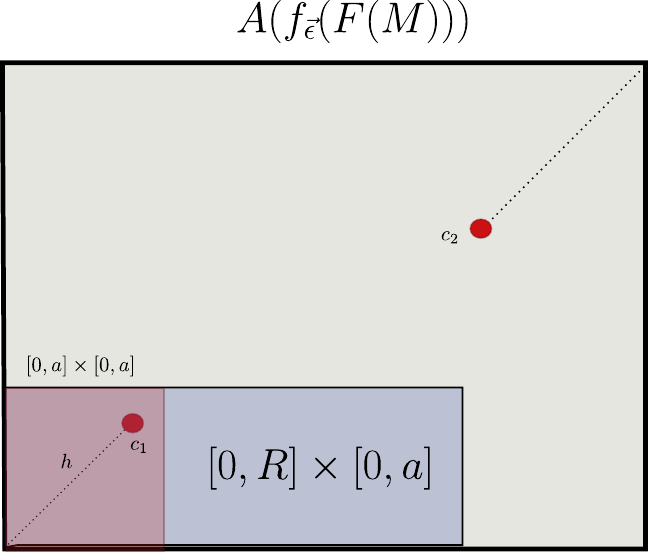}
    \caption{$A(f_{\vec{\epsilon}}(F(M)))$ is the result of applying a suitable integral affine transformation $A$ to the representative of the polytope invariant $f_{\vec{\epsilon}}(F(M))$, where $f_{\vec{\epsilon}}$ and $F$ are the maps of Figure \ref{f.straighteninghomeomorphism}. Let $0<h<a<\frac{R}{2}$. The square $[0,a]\times [0,a]$ fits inside the rectangle $[0,R]\times[0,a]$. The rectangle is inside $A(f_{\vec{\epsilon}}(F(M)))$ and does not intersect the cut associated with the focus-focus value $c_2$. The rectangle $[0,R]\times [0,a]$ represents the symplectic embedding of the product of disks $D^2(R)\times D^2(a)$. Under these conditions the focus-focus fiber $(A\circ f_{\vec{\epsilon}}\  \circ F)^{-1}(c_1)$ is displaceable, since we can displace $D^2(a)\times D^2(a)$ inside $D^2(R)\times D^2(a)$, due to the fact that $a<\frac{R}{2}$, and we can "torically smoothen" the corner introduced by $c_1$.}
    \label{f.intuitiondisplacingfffibers}
\end{center}
\end{figure}
\begin{theorem}
    In a semitoric system a focus-focus fiber can be displaced as long as there exists a symplectic chart with "enough space", in the sense of Figure \ref{f.intuitiondisplacingfffibers}, to accommodate the product of the aforementioned disks. In more detail, if there exists a symplectic embedding, as in Figure \ref{f.intuitiondisplacingfffibers}, of $(D^2(R)\times D^2(a),\omega_0\bigoplus \omega_0)\hookrightarrow (M,\omega)$ such that
    \begin{itemize}
    \item the image of $D^2(a)\times D^2(a)$ contains the focus-focus fiber $F^{-1}(c_1)$,
    \item $h<a<\frac{R}{2}$, where $h$ is the height invariant of the simple focus-focus value $c_1$,
    \end{itemize}
    then the focus-focus fiber is displaceable.
\end{theorem}
This is the content of Section \ref{s.generalization}
and Section \ref{s.displacinggeneralff}, and is proven in Theorem \ref{l.displaceaingff}.

In Section \ref{s.examples} we apply our results to study the (non)displaceability properties of semitoric fibers in $3$ examples:

First, in Section \ref{s.spin} we study the coupled spin-oscillator which is a semitoric system $(M,\omega,F)$ on $M=\mathbb{S}^2\times \mathbb{R}^2$, for which the main result is:
\begin{theorem}
All of the fibers of a coupled spin-oscillator are \textbf{displaceable}. 
\end{theorem}
This is proven in Theorem \ref{t.CSO}.

In Section \ref{s.coupledangularmomenta} we study the coupled angular momenta and its variations. Let $R_1,R_2\in \mathbb{R}^{>0}$ and $t\in [0,1]$ a parameter. Consider the product manifold $M:=\mathbb{S}^2\times \mathbb{S}^2$ with symplectic form $\omega:=-(R_1\omega_{\mathbb{S}^2}\bigoplus R_2\omega_{\mathbb{S}^2})$ where $\omega_{\mathbb{S}^2}$ is the standard symplectic form on $\mathbb{S}^2$.
%Let $(x_1,y_1,z_1,x_2,y_2,z_2)$ be Cartesian coordinates on $M$ induced from the ambient $\R^3 \times \R^3$. 
The \textbf{coupled angular momenta system} is a family of $4$-dimensional completely integrable systems $(M,\omega,F_t:= (L,H_t))$,
%where the smooth functions $L,H_t:M\rightarrow \mathbb{R}$ are given by
%\begin{equation*}
%\begin{cases}
%L(x_1,y_1,z_1,x_2,y_2,z_2):=R_1z_1+R_2z_2, \\
%H_t(x_1,y_1,z_1,x_2,y_2,z_2):=(1-%t)z_1+t(x_1x_2+y_1y_2+z_1z_2)
%\end{cases}
%\end{equation*}
with $R_1<R_2$. The case $R_1>R_2$ is called the \textbf{reverse coupled angular momenta} and the case $R_1=R_2$ is called the \textbf{Kepler problem}. For an appropriate range of the parameter $t$ a focus-focus point always exists.

\begin{theorem}
Let $R_1  \neq R_2$ and $t$ be in the parameter range such that a focus-focus point exists. Then
\begin{itemize}
\item there only exists one nondisplaceable fiber (for which we give an explicit formula in Section \ref{s.coupledangularmomenta}) in the coupled angular momenta system. In particular it is nondisplaceable by symplectomorphisms. 
\item All other fibers are displaceable. 
\item In particular, the focus-focus fiber is displaceable.
\end{itemize}
\end{theorem}
This is proven in Theorem \ref{p.fibersofcoupledangularmomenta}.
Note that the symplectic rigidity problem is sensitive to the choice of $R_1$ and $R_2$. For the Kepler problem, concerning the question of (non)displaceability, we get a different behavior:
\begin{theorem}
    Let $R_1=R_2$ and $t$ be in the parameter range such that a focus-focus point exists. There exists a parameter value $t_0$ such that:
    \begin{itemize}
        \item If $t<t_0$ there exists only a single nondisplaceable fiber, which is nondisplaceable by symplectomorphisms. In particular the focus-focus fiber is displaceable. 
        \item If $t=t_0$ the focus-focus fiber is nondisplaceable. Furthermore, all other fibers are displaceable, and, in particular, the focus-focus fiber is nondisplaceable by symplectomorphisms. 
        \item If $t>t_0$ there exist an infinite amount of nondisplaceable fibers. In particular the focus-focus fiber is nondisplaceable. 
    \end{itemize}
\end{theorem}
This is proven in Theorem \ref{pr.keplerOneNondispl}, Theorem \ref{T.Keplert0} and Theorem \ref{c.t0nondisplaceable}.

In Section \ref{s.octagon} we study the (non)displaceability of the fibers of the toric system with Delzant polytope the octagon $\Delta$ with vertices
\begin{equation*}
    \{(1,0),(0,1),(2,0),(0,2),(1,3),(2,3),(3,1),(3,2)\}.
\end{equation*}
Then we consider the semitoric perturbation of this system, introduced in De Meulenaere \& Hohloch \cite{de2021family}, and study the (non)displaceability properties of the semitoric fibers of this perturbation. Let $(M,\omega,F_t)$ be the semitoric perturbation of the toric system. For $t=\frac{1}{2}$ the system has two focus-focus fibers with two focus-focus points and otherwise $4$ focus-focus fibers with a single focus-focus point. In Appendix \ref{s.computingpolytope} we show that a representative of the polytope invariant, for $t\neq \frac{1}{2}$, is given by the octagon $\Delta$.
\begin{theorem}
    Let $(M,\omega,F_t)$ be the semitoric system induced by the octagon. Then
    \begin{itemize}
        \item For $t=\frac{1}{2}$
        \begin{enumerate}
            \item The fiber over the point $(\frac{3}{2},\frac{3}{2})$ in any representative of the polytope invariant and the focus-focus fibers are nondisplaceable.
            \item All other fibers are displaceable.
        \end{enumerate}
        \item For $t\neq \frac{1}{2}$
        \begin{enumerate}
            \item If the images of the focus-focus values in the representative of the polytope invariant given by the octagon are contained in the region $\{y>2\}\cup \{y<1\}$ then the fibers over the points $\{(1,1),(2,1),(\frac{3}{2},\frac{3}{2}),(1,2),(2,2)\}$ are nondisplaceable. All other fibers are displaceable. In particular the focus-focus fibers are displaceable.
            \item If the images of the focus-focus values in the representative of the polytope invariant given by the octagon are contained in the region $\{y\leq 2\}\cap \{y\geq 1\}$ then the focus-focus fibers and the fiber over the point $\{(\frac{3}{2},\frac{3}{2})\}$ are nondisplaceable. All other fibers are displaceable.
        \end{enumerate}
    \end{itemize}
\end{theorem}
This is proven in several subsequent statements throughout Section \ref{s.octagon}.

If one wants to study (non)displaceability questions in integrable systems the next natural step is to consider so called hypersemitoric systems, or even more general integrable systems. This would mean dealing with fibers containing hyperbolic or degenerate singularities which is beyond the scope of this paper. For examples of hypersemitoric systems see for example Dullin \& Pelayo \cite{dullin2016generating} and Efstathiou \& Hohloch \& Santos \cite{efstathiou2024affine}.

\subsection{Structure of the paper}
\begin{itemize}
\item In Section \ref{s.sectionPrelim} we give the necessary definitions and conventions regarding symplectic manifolds, almost toric fibrations, semitoric systems, the method of probes, and symplectic quasi-states. Furthermore we recall essential results from the literature regarding these subjects. 
\item In Section \ref{s.displacingfocusofocusfibersiwthgeneralized} we generalize McDuff's method of probes and give enough conditions in order to be able to displace focus-focus fibers in semitoric systems. Furthermore, we prove that a focus-focus fiber containing more than a single focus-focus points is nondisplaceable.
\item In Section \ref{s.examples} we study the (non)displaceability properties of $3$ semitoric systems, namely the coupled spin-oscillator, the coupled angular momenta and the semitoric system induced by the octagon. 
\item In the Appendix \ref{s.computingpolytope} we compute a representative of the polytope invariant of the semitoric system induced by the octagon. 
\end{itemize}

\subsection*{Acknowledgments}
The first author was partially and second author fully supported by the FWO-EoS project {\em Beyond symplectic geometry} with UA Antigoon number 45816. Moreover, the first author was also partially supported by the grant "Francqui onderzoekshoogleraar 2023-2026" of the Francqui Foundation with UA Antigoon number 49741.

\section{Preliminaries and background}
\label{s.sectionPrelim}
\subsection{Notions and conventions from symplectic geometry}
A \textbf{symplectic manifold} is a pair $(M,\omega)$ where $M$ is a smooth manifold and $\omega$ is a closed non-degenerate $2$-form. A submanifold $Y$ of $M$ is a $\textbf{Lagrangian submanifold}$ if, at each $p\in Y$, the tangent space $T_pY$ is a Lagrangian subspace of $T_pM$, i.e., $\omega_p|_{T_pY}\equiv 0$ and $\dim(T_pY)=\frac{1}{2}\dim(T_pM)$.

Let $H:M\rightarrow \mathbb{R}$ be a smooth function. Due to the non-degeneracy of $\omega$, there exists a unique vector field $X_H$ on $M$ such that $\omega(X_H,\cdot)=dH$. We call $X_H$ the \textbf{Hamiltonian vector field} of $H$, and its flow $\phi^t_H$ the \textbf{Hamiltonian flow} of $H$.
The \textbf{Poisson bracket} of two functions $f,g\in C^{\infty}(M,\mathbb{R})$ induced by $\omega$ is
    \begin{equation*}
        \{f,g\}:=\omega(X_f,X_g).
    \end{equation*}
Now let $(M,\omega)$ be a symplectic manifold.
We say that a subset $X \subset M$ \textbf{cannot be displaced} from a subset $Y \subset M$ if
    \begin{equation*}
        \phi^1_H(X)\cap \overline{Y}\neq \emptyset,\quad \forall H\in C^{\infty}_{c}([0,1]\times M,\mathbb{R}).
    \end{equation*}
where $\overline{Y}$ denotes the closure of $Y$, and $C^{\infty}_c([0,1]\times M, \mathbb{R})$ denotes the set of smooth compactly supported time-dependent Hamiltonians $H_t$ on $M$. If $X$ cannot be displaced from itself we call it \textbf{nondisplaceable} and otherwise \textbf{displaceable}.

%%%%%%%%%%%%%%%%%%%%%%%%%%%%%%%%%%%%%%%%%%%%%%%%%%%%
%%%%%%%%%% new subsection  %%%%%%%%%%%%%%%%%%%%%%%%%%%

\subsection{Integrable systems}

Let $(M,\omega)$ be a $2n$-dimensional symplectic manifold, and $\{\cdot\ , \cdot\}$ be the Poisson bracket induced by $\omega$.
An \textbf{integrable system} is a triple $(M,\omega,F)$ where $F:M\rightarrow \mathbb{R}^n$ is given by a collection of $n$ real valued smooth functions $f_1,...,f_n$ on $M$ such that:
    \begin{itemize}
        \item $f_1,...,f_n$ are in involution, i.e., $\{f_i,f_j\}=0$ for all $1\leq i,j\leq n$;
        \item $f_1,...,f_n$ are independent, i.e., $X_{f_1},...,X_{f_n}$ are linearly independent almost everywhere on $M$.
    \end{itemize}
We call the map $F:=(f_1,...,f_n):M\rightarrow \mathbb{R}^n$ the \textbf{momentum map} of the integrable system.

%%%%%%%%%%%%%%%%%%%%%%%%%%%%%%%%%%%%%%%%%%%%%%%%%%%%
%%%%%%%%%% new subsection  %%%%%%%%%%%%%%%%%%%%%%%%%%%
%\subsection{Classifications and isomorphism of integrable systems}
%A goal in the research of the symplectic geometry of integrable systems is to construct objects, (such as numbers, functions, polytopes, etc.), which are invariant by isomorphisms, in terms of which a class of integrable systems can be classified up to these isomorphisms. Examples of such classes will be discussed next, such as toric, Section \ref{s.toric}, and semitoric systems, Section \ref{s.semitoric}.
Two integrable systems $(M,\omega,F)$ and $(M',\omega',F')$ are isomorphic if there exists a diffeomorphism $\psi:M\rightarrow M'$ such that $\psi^*\omega'=\omega$ and a diffeomorphism $g:F(M)\rightarrow F'(M')$ such that $F' \circ \psi = g\circ F$.

%%%%%%%%%%%%%%%%%%%%%%%%%%%%%%%%%%%%%%%%%%%%%%%%%%%%
%%%%%%%%%% new subsection  %%%%%%%%%%%%%%%%%%%%%%%%%%%
\begin{quotation}
{\em Henceforth we assume $M$ to be connected and all integrable systems $F=(f_1, \dots, f_n): M \to \mathbb R^n$ are proper with connected fibers.}
\end{quotation}

A point $m\in M$ in which $X_{f_1}(m),...,X_{f_n}(m)$ of $T_mM$ are linearly dependent is said to be \textbf{singular} or a \textbf{singularity} of $F$. Otherwise we call it \textbf{regular}.
The \textbf{rank} of $m$ is given by $\dim \Span_\R \{ X_{f_1}(m),...,X_{f_n}(m)\}$.
If a fiber $F^{-1}(c)$ for some $c \in \R^m$ contains a singularity we call it \textbf{singular} and otherwise \textbf{regular}.

\subsection{Action-angle coordinates}
Denote a $k$-dimensional torus by $\mathbb{T}^k$.
One important theorem about the symplectic geometry of integrable systems is the existence of so-called {\em action-angle coordinates}, see Arnold \cite{arnol1963theorem}: each regular fiber, in addition to being diffeomorphic to $\mathbb{T}^n$, can be seen as sitting inside of the cotangent bundle $T^*\mathbb{T}^n$ as the zero section and, in a neighborhood of the fiber, the integrable system can be written in a symplectic normal form $G:T^*\mathbb{T}^n\rightarrow \mathbb{R}^n$:
\begin{theorem}
\label{t.AL}
({Liouville-Arnold-Mineur Theorem}) Let $(M,\omega,F)$ be an integrable system and let $c\in F(M)$ be a regular value. If $\Lambda_c:=F^{-1}(c)$ is a regular, compact and connected fiber, then there exist neighborhoods $U\subset F(M)$ of $c$ and $V\subset \mathbb{R}^n$ of the origin, such that for
\begin{equation*}
\mathcal{U}:= \coprod_{r\in U}F^{-1}(r)\quad  \text{and} \quad \mathcal{V}:=\mathbb{T}^n\times V \subset T^*\mathbb{T}^n
\end{equation*}
we have that $(\mathcal{U},\omega|_{\mathcal{U}},F|_{\mathcal{U}})$ and $(\mathcal{V},\omega_0|_{\mathcal{V}},G|_{V})$ are isomorphic integrable systems, where $\omega_0$ is the standard symplectic form on $T^*\mathbb{T}^n$.
\end{theorem}

By the Liouville-Arnold-Mineur Theorem, Theorem \ref{t.AL}, a preimage of a regular value is diffeomorphic to $\mathbb{T}^n$. 
Moreover, $\omega$ vanishes along regular fibers $F^{-1}(c)$, hence $F^{-1}(c)$ is a Lagrangian submanifold, traditionally called a \textbf{Liouville torus}. One often refers to an integrable system $(M,\omega,F)$ as a singular Lagrangian fibration where {\em singular} emphasizes that $F$ may have singular fibers of various kinds: for example tori of dimension $m\in \{0,...,n-1\}$, pinched tori etc.

%%%%%%%%%%%%%%%%%%%%%%%%%%%%%%%%%%%%%%%%%%%%%%%%%%%%
%%%%%%%%%% new subsection  %%%%%%%%%%%%%%%%%%%%%%%%%%%

\subsection{Linearization of non-degenerate singularities}
\label{sec:locNormalForm}

    A singular point $p\in M$ of rank zero of an integrable system $(M,\omega,F=(f_1,...,f_n))$ is \textbf{non-degenerate} if the Hessians $d^2f_1(p),...,d^2f_n(p)$ span a Cartan subalgebra of the Lie algebra of quadratic forms on $T_pM$. We refer to Bolsinov \& Fomenko \cite{bolsinov2004integrable} for the precise definition of non-degenerate points of higher rank.% In Section \ref{s.nondegeneratefour} we recall it for the case $\dim(M)=4$.

Non-degenerate singularities are in fact linearizable, more precisely, there exists a local normal form which is based on the works of R\"ussmann \cite{russmann1964verhalten}, Vey \cite{vey1978certains}, Colin de Verdi\'ere $\&$ Vey \cite{de1979lemme}, Eliasson \cite{eliasson1990normal,eliassonphdthesis}, Dufour and Molino \cite{dufour1991compactification}, Miranda \cite{miranda2014integrable,mirandaphdthesis}, Miranda $\&$ Zung \cite{miranda2004equivariant}, Miranda $\&$ \vungoc \ \cite{miranda2005singular}, \vungoc\ $\&$ Wacheux \cite{vu2013smooth}, and Chaperon \cite{chaperon2013normalisation}: there are symplectic local coordinates $(x_1,...,x_n,\xi_1,...,\xi_n)$ near the singular point $m$ in which $m=(0,...,0)$ corresponds to the origin, the symplectic form assumes the expression $\omega=\sum_{i=1}^n dx_i\wedge d\xi_i$ and there exist functions $q_1,...,q_n$ of $(x_1,...,x_n,\xi_1,...,\xi_n)$ such that the integrable system $F=(f_1,...,f_n)$ satisfies  $\{f_i,q_j\}=0$ for all indices $1 \leq i,j \leq n$, where $q_i$ is one of the following functions:
\begin{itemize}
    \item \textbf{elliptic} type: $q_i=\frac{x_i^2+\xi_i^2}{2}$,
    \item \textbf{hyperbolic} type: $q_i:=x_i\xi_i$,
    \item \textbf{regular} type: $q_i=\xi_i$,
    \item \textbf{focus-focus} type: $q_i=x_i\xi_{i+1}-x_{i+1}\xi_i$ and $q_{i+1}=x_i\xi_i+x_{i+1}\xi_{i+1}$.
\end{itemize}
If there are no components of hyperbolic type then
\begin{equation*}
    (F-F(m))\circ \phi = g\circ (q_1,...,q_n)
\end{equation*}
where $\phi$ is the inverse of the local coordinates $(x_1,...,x_n,\xi_1,...,\xi_n)$, and $g$ is a diffeomorphism from a small neighborhood of $(0,...,0)$ into another such neighborhood such that $g(0,...,0)=(0,...,0)$. In this paper, we will mostly work with singularities without components of hyperbolic type.
%One reason for this is that their appearance makes it difficult to construct global symplectic invariants.

\begin{quotation}
Throughout this paper, we always assume that the singularities of $F:M\rightarrow \mathbb{R}^n$ are \textbf{non-degenerate}.
\end{quotation}

\subsection{Toric systems}
\label{s.toric}

The simplest, non-trivial, type of integrable systems are the following ones.

\begin{definition}
An integrable system $(M,\omega,F=(f_1,...,f_n))$ on a symplectic $2n$-dimensional manifold $(M, \omega)$ is \textbf{toric} if the Hamiltonian vector fields $X_{f_1},...,X_{f_n}$ generate periodic flows of the same period (in our convention $2\pi$) and the action of $\mathbb{T}^n$ on $M$ induced by these flows is effective (i.e. only the identity acts trivially).
\end{definition}

%Alternatively, one can define a toric system via so-called Hamiltonian group actions: If a Lie group $G$ acts on a symplectic manifold $(M,\omega)$ then each element $\xi$ of the Lie algebra $\mathfrak{g}:= Lie(G)$ induces a vector field $X_{\xi}$, called the infinitesimal action.
%\begin{definition}
 %   A group action on a symplectic manifold $(M,\omega)$ is \textbf{Hamiltonian} if there is a Lie algebra homomorphism from $\mathfrak{g}$ to the smooth functions on $M$, with the Poisson bracket, taking $\xi$ to $f_{\xi}$, where $f_{\xi}$ is such that $\omega(X_{\xi},\cdot)=df_{\xi}$.
%\end{definition}
%If a group action is Hamiltonian then there is an associated momentum map $\mu:M\rightarrow \mathfrak{g}^{*}$ defined implicitly by $\langle \xi, \mu(x) \rangle = f_{\xi}(x)$ for $\xi\in \mathfrak{g}$ and $x\in M$. Thus we can think of a toric system as a Hamiltonian torus action on a symplectic manifold.

\begin{quotation}
 {\em Unless otherwise stated, we will usually consider toric integrable systems on closed connected manifolds.}
\end{quotation}

The periodicity condition implies that the singularities of toric integrable systems cannot have focus-focus or hyperbolic components, that is, if $m=(0,...,0)$ and $\omega=\sum_{i=1}^n dx_i\wedge d
\xi_i$, then the integrable system is locally in a neighborhood of $m$ of the form
\begin{equation*}
    F(x_1,...,x_n,\xi_1,...,\xi_n)=\left(\frac{x_1^2+\xi_1^2}{2},...,\frac{x_k^2+\xi_k^2}{2},\xi_{k+1},...,\xi_n\right)
\end{equation*}
for some $0 \leq k \leq n$.
Toric integrable systems have connected fibers, see Atiyah \cite{atiyah1982convexity}.
%, a fact known as Atiyah's connectivity.
In particular, all fibers of a toric system $F$ are diffeomorphic to tori of varying dimensions $\mathbb{T}^k$ with $ k\in \{0,...,n\}$. Note that this is not the case for general integrable systems.

One of the fundamental theorems of equivariant symplectic geometry, due to Atiyah \cite{atiyah1982convexity} and Guillemin \& Sternberg \cite{guillemin1982convexity} says that, if $M$ is compact and connected, the image $F(M)$ is a convex polytope in $\mathbb{R}^n$, obtained as the convex hull of the images of the fixed points of the Hamiltonian action of the $n$-torus on $M$. This polytope has the property that if two toric integrable systems are isomorphic then their associated images coincide, possibly up to translations and composition with a matrix in $\text{GL}(n,\mathbb{Z})$. Furthermore, Delzant \cite{delzant1988hamiltoniens} showed that the convex polytopes obtained as images of toric integrable systems are of a special type: they are 
\begin{itemize}
\item {\bf simple}, i.e., there are precisely $n$ edges meeting at each vertex;
\item {\bf rational}, i.e., the slopes of the edges are rational;
\item  and {\bf smooth}, i.e., normal vectors to the facets meeting at each vertex form a basis of the integral lattice.
\end{itemize}
Such polytopes are called \textbf{Delzant polytopes}. Moreover, Delzant proved that such polytopes are in bijective correspondence with toric integrable systems on compact manifolds in the following sense:
\begin{itemize}
    \item Uniqueness: two systems $(M,\omega,F)$ and $(M',\omega',F')$ are isomorphic if and only if they have the same convex polytope as image of the momentum map (up to translations and $\text{GL}(n,\mathbb{Z})$ transformations).
    \item Existence: from any Delzant polytope $\Delta$ in $\mathbb{R}^n$ one can construct a toric integrable system $(M,\omega,F)$ on a compact connected symplectic $2n$-dimensional manifold with $F(M)=\Delta$. This construction uses the method of symplectic reduction.
\end{itemize}
Since $\Delta=F(M)$ classifies $(M,\omega,F)$, one can learn everything about $F$, up to isomorphisms, from $\Delta$. In particular the fiber structure of $F$ can be read off from the polytope $\Delta$: the fiber of $F$ over $p\in \Delta$ is diffeomorphic to a $k$-dimensional torus, where $k$ is the dimension of the lowest dimensional face of $\Delta$ that $p$ belongs to.

\subsection{Almost toric manifolds}
\label{s.almosttoricmanifolds}
So called almost toric manifolds, as introduced by Symington \cite{symington2002four}, are a generalization of closed toric manifolds and will be defined in more detail below.
\begin{definition}
    A locally trivial fibration of a $2n$-dimensional symplectic manifold $(M,\omega)$ is called a \textbf{regular Lagrangian fibration} if the fibers are Lagrangian submanifolds. More generally, a projection $\pi:(M,\omega)\rightarrow B$, where $B$ is an $n$-dimensional manifold with boundary, is a \textbf{Lagrangian fibration} if it restricts to a regular Lagrangian fibration over the open dense set $B_r\subset B$, defined as the set of regular values of $\pi$. 
\end{definition}
\begin{definition}
     By a lattice in a vector bundle of rank $n$ we mean a smooth varying lattice isomorphic to $\mathbb{Z}^n$ in each fiber of the bundle. An \textbf{integral affine structure} $\mathcal{A}$ on a $n$-manifold $B$ is a lattice in the tangent bundle of $B$.
A manifold admitting such a structure is called an \textbf{integral affine manifold}.
\end{definition}

The integral affine structure on the base of a regular Lagrangian fibration arises from a natural action of the cotangent bundle of the base on the total space: any $\alpha \in T^*B$ defines a vertical vector field $X_{\alpha}$ symplectically dual to $\alpha$, so for $x\in M$ the action of $\alpha \in T^*B$ is given by $\alpha \cdot x =\phi_{\alpha}(x)$,
where $\phi_{\alpha}$ is the time-1 map of $X_{\alpha}$. The elements of the cotangent bundle of $B$ that act trivially, i.e., the isotropy group of this action, form a lattice $\Lambda^*$.
The dual lattice $\Lambda$ in the tangent bundle defines the integral affine structure. We denote by $(\mathbb{R}^n,\mathcal{A}_0)$ the standard integral affine structure of $\mathbb{R}^n$ induced by $\mathbb{Z}^n$.

\begin{definition}
   Let $B$ be an integral affine manifold with integral affine structure $\mathcal{A}$ and $(M,\omega)$ a $2n$-dimensional symplectic manifold. A Lagrangian fibration $\pi:(M,\omega)\rightarrow B$ is a \textbf{toric fibration} if there exists a Hamiltonian $n$-torus action and an immersion $\Phi:(B,\mathcal{A})\rightarrow (\mathbb{R}^n,\mathcal{A}_0)$ such that $\Phi\circ \pi$ is the corresponding momentum map. Denote by $\mathcal{S}$ the following stratification: the $l$-stratum  of $(B,\mathcal{S})$ is the set of points $b\in B$ such that the fiber $F_b:=\pi^{-1}(b)$ is a torus of dimension $l$.  We call the triple $(B,\mathcal{A},\mathcal{S})$ a \textbf{toric base}.  The \textbf{reduced boundary} of a toric base $(B,\mathcal{A},\mathcal{S})$, denoted by $\partial_RB$, is the set of points in $B$ that belong to the lower-dimensional strata, i.e., $k$-strata for $k<n$.
\end{definition}

From a toric base $(B,\mathcal{A},\mathcal{S})$ it is possible to reconstruct a toric fibration and a symplectic manifold by starting with a regular Lagrangian fibration over the base $(B,\mathcal{A})$ and then collapsing certain fibers to lower dimensional tori so as to obtain the stratification $\mathcal{S}$. The collapsing of the fibers is achieved via \textbf{boundary reduction}, see Symington \cite[Proposition $3.8$]{symington2002four} for more details.

%\begin{proposition}
 %   ({Symington \cite[Proposition 3.8]{symington2002four}}) Let $(M,\omega)$ be a symplectic manifold with boundary such that a smooth component $Y$ of $\partial M$ is a circle bundle over a manifold $\Sigma$. Suppose also that the tangent vectors to the circle fibers lie in the kernel of $\omega|_{Y}$. Then there exists a pair $(M',\omega')$, where $M'$ is a manifold and $\omega'$ a $2$-form, a projection $\rho_{Y}:=\rho :(M,\omega)\rightarrow (M',\omega')$ and an embedding $\phi:\Sigma\rightarrow M'$ such that $\rho(Y)=\phi(\Sigma)$, $\rho|_{M-Y}$ is a symplectomorphism onto $M' \setminus \phi(\Sigma)$ and $\Phi(\Sigma)$ is a symplectic submanifold.
%\end{proposition}

%We call $(M',\omega')=\rho(M,\omega)$ the \textbf{symplectic boundary reduction} of $(M,\omega)$ along $Y$.

%\begin{definition}
%    Given a toric fibration $\pi:(M,\omega)\rightarrow (B,\mathcal{A},\mathcal{S})$, the \textbf{boundary recovery} is the unique Lagrangian fibered manifold $(B\times \mathbb{T}^n,\omega_0)$ that yields $(M,\omega)$ via boundary reduction.
%\end{definition}

%The geometric structure of $B$ and the stratification $\mathcal{S}$ together determine what circles get collapsed during boundary reduction of $(B\times \mathbb{T}^n,\omega_0)$. Recall that $B_r\subset B$ is the subset of regular values.

\begin{definition}
    Let $\pi:(M,\omega)\rightarrow (B,\mathcal{A},\mathcal{S})$ be a toric fibration and $\gamma$ the image in $B$ of a compact embedded curve with one endpoint $b_1$ in the $(n-1)$-dimensional stratum of $\partial_RB$ and such that $\gamma\setminus\{b_1\}\subset B_r$.
    Let $b_2\in B$ be the other endpoint of $\gamma$. The \textbf{collapsing class} with respect to $\gamma$ for the smooth component of $\partial_RB$ containing $b_1$ is defined as the class $a\in H_1(F_{b_2},\mathbb{Z})$ that spans the kernel of $i_*:H_1(F_{b_2},\mathbb{Z})\rightarrow H_1(\pi^{-1}(\gamma),\mathbb{Z})$ where $i$ is the inclusion map $i:F_{b_2}\rightarrow \pi^{-1}(\gamma)$. See Figure \ref{f.collapsingclass} for topological intuition.
    Accordingly, the \textbf{collapsing covector} with respect to $\gamma$ is the primitive covector $v^*\in T^*_{b_2}B$ that determines vectors $v(x)\in T^{vert}_xM:=\ker(d_x\pi)$ for each $x\in F_{b_2}$ such that the flow lines of this vector field represent $a$. (Notice $a$ and $v^*$ are well-defined up to sign and have coefficients in $\mathbb{Z}$.)
\end{definition}

\begin{figure}
\begin{center}
    \includegraphics[width=0.5\textwidth]{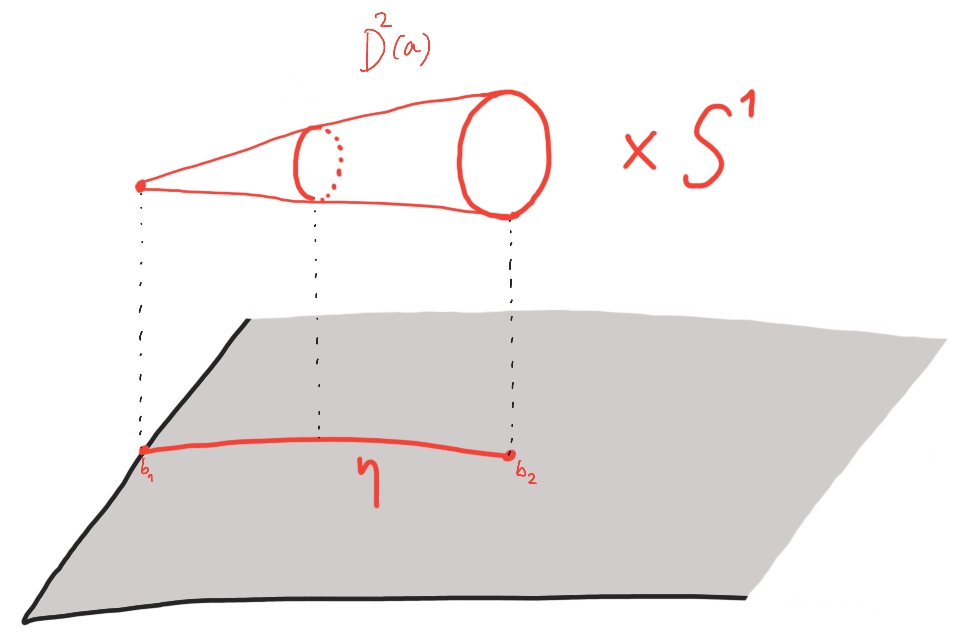}
    \caption{Let $D^2(a)$ represent the disk of area $a$, for $a>0$. The preimage of the line $\eta$ by the fibration $\pi$ is $D^2(a)\times S^1$, while the fiber $\pi^{-1}(b_2)$ corresponds to $\partial(D^2(a))\times S^1$. Furthermore, the collapsing class is the homology class induced by $\partial(D^2(a))$.}
    \label{f.collapsingclass}
\end{center}
\end{figure}

\begin{remark}
    If $v$ is a primitive integral vector that is normal with respect to the standard Euclidean metric to a smooth component of the reduced boundary of $(B,\mathcal{A}_0,\mathcal{S})\subset (\mathbb{R}^2,\mathcal{A}_0)$ then $v^*$ is a collapsing covector for that component.
\end{remark}

\begin{definition}
    An \textbf{almost toric fibration} of a symplectic $2n$-manifold $(M,\omega)$ is a Lagrangian fibration $\pi:(M,\omega)\rightarrow B$ such that any critical point of $\pi$ has a Darboux neighborhood (i.e.\ local coordinates as in Section \ref{sec:locNormalForm}) in which the components of the projection $\pi=(\pi_1,...,\pi_k,\pi_{k+1},...,\pi_{n})$ are of the form $\pi_j(x,y)=x_j$ for $j\leq k$ for some $k\leq 2n$ and otherwise are of the following two forms:
    \begin{itemize}
        \item \textbf{elliptic} (or \textbf{toric}) : $\pi_j(x,y)=(x_j^2+y_j^2)$,
        \item \textbf{nodal} (or \textbf{focus-focus}): $(\pi_i,\pi_{i+1})(x,y)=(x_iy_{i+1}-x_{i+1}y_i,\ x_iy_i+x_jy_j)$.
    \end{itemize}
    An \textbf{almost toric manifold} is a symplectic manifold equipped with an almost toric fibration.
\end{definition}

%Therefore a semitoric system induces an almost toric fibration on the underlying symplectic manifold.

Assigning points in the base of an almost toric fibration to strata according to the dimension of their preimage yields a \textbf{stratification} $\mathcal{S}$ of the base. When $(M,\omega)$ has dimension $4$ the images of nodal singular points (nodes) are isolated points that belong to the top dimensional stratum. Let $\Sigma$ be the codimension two set of points containing a nodal singularity in their preimage.

By the integral affine structure $\mathcal{A}$ on the base of an almost toric fibration we mean the affine structure defined on $B\setminus\Sigma$.
If an affine structure $\mathcal{A}$ and stratification $\mathcal{S}$ are induced from an almost toric fibration, we call the triple $(B,\mathcal{A},\mathcal{S})$ an \textbf{almost toric base}.

%An essential way in which the base $(B,\mathcal{A},\mathcal{S})$ influences the topology of the total space is by capturing the monodromy. Specifically, the topological monodromy of the torus fibration over the regular values $B_r\subset B$ is determined by the affine monodromy in the base, i.e., the lattice in $TB_r$.

Let $\text{Aff}(n,\mathbb{Z})$ be the group of integral affine transformations, i.e., maps $T:\mathbb{R}^n\rightarrow \mathbb{R}^n$ of the form $T(x)=A x+ b$, with $A\in GL(n,\mathbb{Z})$ and $b\in \mathbb{R}^n$.

Let $B_r\subset B$ be the set of regular values of the almost toric fibration $B$. Recall that $B_r$ is an integral affine manifold.
The \textbf{affine monodromy} of $B_r$ is defined as follows. Let $\Lambda$ be the lattice in $TB_r$. Choose a point $b\in B_r$, identify $(T_bB_r,\Lambda_b)$ with $(\mathbb{R}^n,\mathbb{Z}^n)$ and for each element $\alpha \in \pi_1(b,B_r)$ choose a representative $\gamma_{\alpha}:I:=[0,1] \rightarrow B_r$.
The monodromy representation relative to these choices is denoted by $\Psi_{B_r}:\pi_1(b,B_r)\rightarrow \text{Aff}(n,\mathbb{Z})$ and defined as follows: for all $\alpha \in \pi_1(b,B_r)$, one defines $\Psi_{B_r}(\alpha)$ as the automorphism of $(\mathbb{R}^n,\mathbb{Z}^n)$ such that $\gamma_{\alpha}^*(TB_r,\Lambda)$ is isomorphic to $I\times (\mathbb{R}^n,\mathbb{Z}^n)/\left((0,p)\simeq (1,\Psi_{B_r}(\alpha)(p))\right)$ for $ p\in \mathbb{R}^n$.
The monodromy is the equivalence class of monodromy representations relative to different points in $B_r$ and choices of identifications of $(T_b{B_r},\Lambda_b) $ with $(\mathbb{R}^n,\mathbb{Z}^n)$. Analogously, for an almost toric fibration $\pi:(M,\omega)\rightarrow B$ one can define the \textbf{topological monodromy} of $\alpha\in \pi_1(b,B_r)$ by considering $\pi_1(\pi^{-1}(b),\mathbb{Z})\cong \mathbb{Z}^n$ instead of $\Lambda_b$. 

\begin{definition}
    \label{d.vanishingclass}
    Let $\pi:(M,\omega)\rightarrow B$ be an almost toric fibration that has a node at $s\in B$.
    Let $\eta$ be the image in $B$ of an embedded curve with endpoints at $s$ and at a regular point $b\in B_r$ such that $\eta\setminus\{s\}\subset B_r$ contains no other nodes. Associated to $s$ and $\eta$ is the (well defined up to sign) \textbf{vanishing class} in $H_1(F_b,\mathbb{Z})$, namely the class whose representatives bound a disk in $\pi^{-1}(\eta)$. See Figure \ref{f.vanishingclass} for topological intuition.
    The \textbf{vanishing covector} $v^*\in T_b^*B$ is the covector (defined up to sign) that determines vectors $v(x)\in T_x^{vert}M$ for each $x\in \pi^{-1}(b)$ such that the flow lines of this vector field represent the vanishing class.
\end{definition}
\begin{figure}
\begin{center}
    \includegraphics[width=0.5\textwidth]{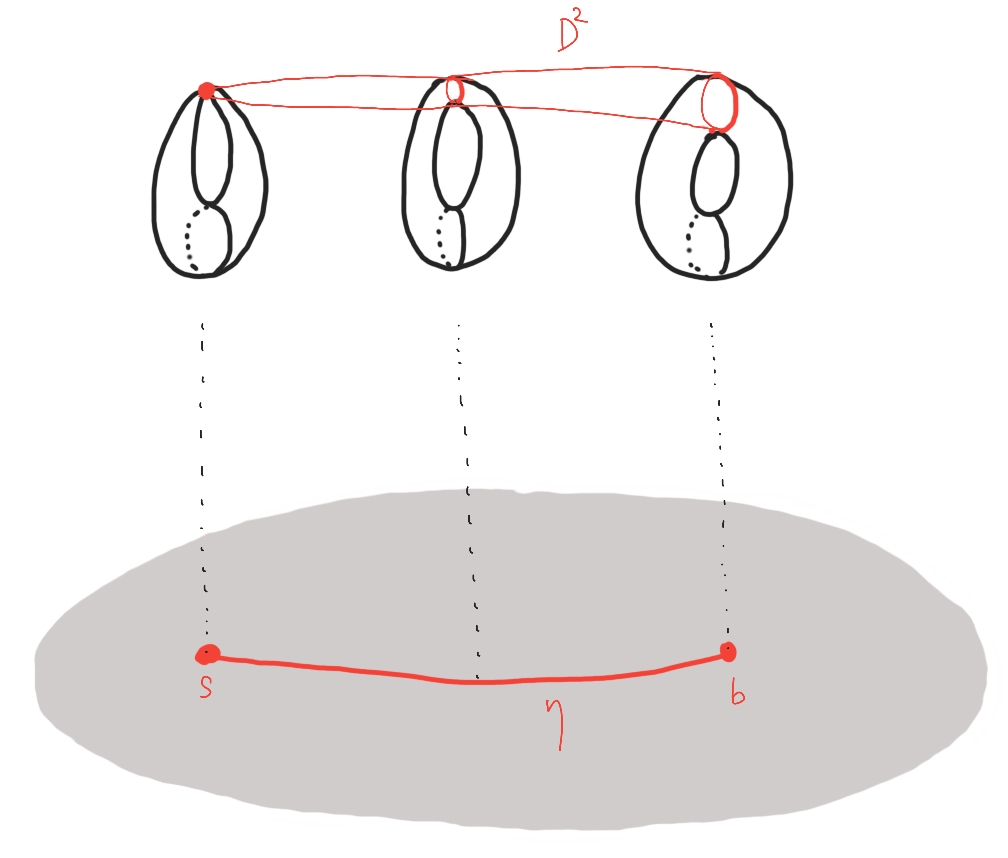}
    \caption{The red loop in the torus fiber over $b$ is the vanishing class in $H_1(F_b,\mathbb{Z})$.}
    \label{f.vanishingclass}
\end{center}
\end{figure}

\begin{lemma}
    ({Symington, \cite[Lemma 4.10]{symington2002four}}) Using the notation from Definition \ref{d.vanishingclass} suppose $\gamma:\mathbb{S}^1\rightarrow B$ is a positively oriented loop based at $b\in B$ that is the boundary of a closed neighborhood of $s\in B$ containing $\eta\subset B$. Then the vanishing class, up to rescaling, is the unique class that is preserved by the monodromy along $\gamma$.
\end{lemma}

Let $\pi:(U,\omega)\rightarrow B$ be a Lagrangian fibered neighborhood of a nodal fiber in a $4$-manifold where $B$ is connected and (the value of) the node is $s\in B$. Let $\mathcal{A}$ be the affine structure on $B\setminus\{s\}$.
The affine monodromy is non-trivial, due to the presence of the node (see Symington, \cite{symington2002four} Section $4.3$ and Zung \cite{MR1454718}). Hence there is no immersion of $(B\setminus\{s\},\mathcal{A})$ into $(\mathbb{R}^2,\mathcal{A}_0)$. However, there is an affine immersion if one removes a curve from $B$ that connects the (value of the) node to the boundary since the remaining set is simply connected.
We define the base diagram for the neighborhood of a node, as the image of an affine immersion of the complement of such a curve.
%This is analogous to the definition of the polytope invariant for a semitoric system.

\begin{definition}Let $\pi:(M,\omega)\rightarrow B$ be a an almost toric fibration. A $1$-dimensional submanifold of $(B,\mathcal{A})$ is \textbf{affine linear} if at every point it has a tangent vector in $\Lambda$. 
The \textbf{eigenline} $L\subset B$ through a node $s$ is the unique maximal affine linear immersed one-manifold through the node for which there is a regular point $b\in L$, arbitrarily close to $s$, such that the affine monodromy along an arbitrarily small loop around $s$ and based at $b$ preserves $T_bL\subset T_bB$. An \textbf{eigenray} is either of the two maximal affine linear submanifolds that has an endpoint at the node and is a subset of the eigenline.
\end{definition}

\begin{remark}
    Consider a base diagram for an almost toric base $(B,\mathcal{A},\mathcal{S})$ of an almost toric fibration. If a primitive integral vector $v$ is perpendicular to the eigenline through a node then $v^*$ is a covector that defines the vanishing class of the fiber, with respect to the eigenline.
\end{remark}

%%%%%%%%%%%%%%%%%%%%%%%%%%%%%%%%
%%%%%%%%%%% new subsection

\subsection{Nodal trades and slides}
\label{s.nodaltrade}
Henceforth we work with $4$-dimensional symplectic manifolds $(M,\omega)$, therefore appearing bases $B$ are $2$-dimensional in this subsection.
Nodal trades and slides are surgery operations that change an almost toric fibration of a symplectic $4$-manifold into another almost toric fibration of the same symplectic manifold. In a nodal trade the essential idea is that one can trade an elliptic-elliptic fiber for a nodal fiber, and vice versa under the appropriate conditions. A nodal slide should be thought of as a one-parameter family of almost toric bases in which a node moves in the base along its eigenline.

\begin{definition}
\label{d.nodalslide}
    Two almost toric bases $(B,\mathcal{A}_i,\mathcal{S}_i), i=1,2$, are related by a \textbf{nodal slide} if there is a curve with image $\gamma \subset B$ such that $(B\setminus\gamma,\mathcal{A}_1,\mathcal{S}_1)$ and $(B\setminus\gamma,\mathcal{A}_2,\mathcal{S}_2)$ are isomorphic (i.e.\ there exists a map $\phi:(B\setminus\gamma,\mathcal{A}_1,\mathcal{S}_1)\rightarrow (B\setminus\gamma,\mathcal{A}_2,\mathcal{S}_2)$ that preserves the affine structures and stratifications), and for each $i$, the curve $\gamma$ contains one node of $(B,\mathcal{A}_i,\mathcal{S}_i)$ and $\gamma$ belongs to the eigenline through that node.
\end{definition}

\begin{proposition}
\label{p.nodalslide}
    ({Symington, \cite[Proposition 6.2]{symington2002four}}) If two bases are related by a nodal slide then they define the same manifold equipped with isotopic symplectic structures. Furthermore, the manifolds are fiberwise symplectomorphic on the complement of a compact set belonging to the preimage of the eigenline. 
\end{proposition}

\begin{lemma}
\label{l.nodaltrade}
    (Symington, \cite[Lemma 6.3]{symington2002four}) Let $(B,\mathcal{A},\mathcal{S})$ be an almost toric base and $R\subset B$ be the image of an embedded eigenray that connects a node $s\in B$ with a point $b$ in an edge $E\subset \partial_{R}B$ such that there are no other nodes on $R$.
    Let $v^*$, $w^*\in T_b(B)^*$ be the vanishing and collapsing covectors. If $v^*$ and $w^*$ span $\Lambda_b^*$ then there is an almost toric base $(B,\mathcal{A}',\mathcal{S}')$ such that:
    \begin{itemize}
        \item $(B,\mathcal{A}',\mathcal{S}')$ contains one node less than $(B,\mathcal{A},\mathcal{S})$,
        \item $(B\setminus R,\mathcal{A}',\mathcal{S}')$ is isomorphic to $(B\setminus R,\mathcal{A},\mathcal{S})$,
        \item in $(B,\mathcal{A}',\mathcal{S}')$, the intersection of $R$ and $\partial_RB$ is a vertex.
    \end{itemize}
    In this situation we say that the two almost toric basis $(B,\mathcal{A},\mathcal{S})$ and $(B,\mathcal{A}',\mathcal{S}')$ are related by a \textbf{nodal trade}.
\end{lemma}

\begin{theorem}
    ({Symington \cite[Theorem 6.5]{symington2002four}})
    Suppose two almost toric bases that are related by a nodal trade. Then the symplectic manifolds induced by the almost toric bases are symplectomorphic. In fact, their symplectic structures are isotopic.
\end{theorem}

%%%%%%%%%%%%%%%%%%%%%%%%%%%%%%%%%%%%%%%%%%%%%%%%%%%%
%%%%%%%%%% new subsection  %%%%%%%%%%%%%%%%%%%%%%%%%%%
\subsection{Semitoric systems}
\label{s.semitoric}
The class of toric systems can be generalized to semitoric systems, which are a special case of an almost toric fibration.
\begin{definition}
\label{d.semitoric}
    An integrable system 
    \begin{equation*}
        (M,\omega,F=(f_1,...,f_n):M\rightarrow \mathbb{R}^n)
    \end{equation*}
    is \textbf{semitoric} if the Hamiltonian vectors fields $X_{f_1},...,X_{f_{n-1}}$ generate periodic flows of the same period (in our convention $2\pi$), if the action of $\mathbb{T}^{n-1}$ induced on $M$ by these flows is effective and if the singularities of $F$ are non-degenerate and do not have hyperbolic components. If $M$ is not compact the integrals $f_1,...,f_{n-1}$ are required to be proper.
\end{definition}
The singular fibers of these systems are either points, circles or tori with a finite number of pinches. Note that this last type of fiber does not appear in toric systems.

If $M$ is $4$-dimensional, semitoric systems were classified, first under some conditions, by Pelayo $\&$ \vungoc\ \cite{pelayo2009semitoric, pelayo2011constructing} and then in full generality by Palmer \& Pelayo \& Tang \cite{palmer2019semitoric}. The distinguishing property of a semitoric system is that all but maybe one of the integrals generate periodic flows. So toric systems are a special case of semitoric ones. We will consider examples in Sections \ref{s.spin} \& \ref{s.coupledangularmomenta} \& \ref{s.octagon}.

A simple semitoric system is a semitoric system where each fiber contains at most one focus-focus singularity. Such systems are determined up to isomorphism, according to Pelayo $\&$ \vungoc\ \cite{pelayo2009semitoric}, by the following $5$ symplectic invariants:
\begin{enumerate}
    \item The \textbf{number of focus-focus singularities}, denoted by $n_{FF}$.
    \item The \textbf{Taylor series invariant}, $n_{FF}$ formal Taylor series in two variables describing the foliation around each focus-focus singular fiber.
    \item The \textbf{polytope invariant}, a family of weighted rational convex polytopes (generalizing the Delzant polytope of toric systems).
    \item The \textbf{height invariant}, given by $n_{FF}$ numbers corresponding to the height of the focus-focus critical values in the rational convex polytopes of the polytope invariant.
    \item The \textbf{twisting index invariant}, given by $n_{FF}$ integers measuring how twisted the system is around singularities from a `toric point of view'.
\end{enumerate}

Let $(M,\omega,(f_1, f_2))$ and $(\tilde{M},\tilde{\omega},(\tilde{f}_1, \tilde{f}_2))$ be two semitoric systems. We say that they are \textbf{isomorphic} as semitoric systems if there exists a symplectomorphism $\psi:M\rightarrow \tilde{M}$ and a smooth map $g:\mathbb{R}^2\rightarrow \mathbb{R}$ such that $(\tilde{f}_1, \tilde{f}_2) \circ \psi =(f_1,g(f_1, f_2))$ and $\partial_2 g >0$. 

The classification result for semitoric systems has two aspects:
\begin{itemize}
\item Two semitoric systems are isomorphic if and only if they have the same list of symplectic invariants
\item Given any admissible list of invariants, a semitoric system with such invariants can be constructed.
\end{itemize}

Since we will later work with (representatives of) the polytope invariant, let us consider it now in more detail. Let
$$\{c_i=(x_i,y_i) \mid i=1,...,m_f\}\in \mathbb{R}^2$$
be the set of focus-focus critical values, ordered in such a way that $x_1\leq x_2\leq ... \leq x_{m_f}$ and let $B_r$ be the set of regular values in $B:=F(M)$. For $i \in \{1, \dots, m_f\}$ and $\epsilon \in \{-1,+1\}$, define $\mathcal{L}_i^{\epsilon}$ to be the vertical ray starting at $c_i$ and going to $ \pm {\infty}$ depending on the sign of $\epsilon$, i.e., $\mathcal{L}_i^{\epsilon}=\{(x_i,y) \mid \epsilon y\geq \epsilon y_i\}$.
Given $\vec{\epsilon}=(\epsilon_1,...,\epsilon_{m_f})\in \{-1,+1\}^{m_f}$, we define the line segment $l_i:=F(M)\cap \mathcal{L}_i^{\epsilon_i}$ and we set
\begin{equation*}
    l^{\vec{\epsilon}}=\cup_{i}l_i
\end{equation*}
where in addition we decorate each $l_i$ with the multiplicity $\epsilon_ik_i$, where $k_i$ is the number of focus-focus points in the fiber $F^{-1}(c_i)$. If several $c_i$ have the same $x_i$ coordinate, then $l_i$ is the union of all corresponding segments. Given $c\in l_i$, define $k(c):=\sum_{c_j}\epsilon_jk_j$ where the sum runs over all focus-focus values $c_j$ such that $c_i\in l_j$.

Denote by $\mathcal{T}$ the subgroup of $\text{Aff}(2,\mathbb{Z})$ which leaves a vertical line, with orientation, invariant. In other words an element of $\mathcal{T}$ is a composition of a vertical translation and an element of $\{T^k,k\in \mathbb{Z}\}\subset GL(2,\mathbb{Z})$, where
\begin{equation*}
T^k:= \begin{bmatrix}
1 & 0 \\
k & 1 
\end{bmatrix}.
\end{equation*}

\begin{theorem}({\vungoc, \cite[Theorem 3.8]{vu2007moment}})
\label{th.straigheningHomeo}
Using the notation from above for a semitoric system $(M,\omega,(f_1,f_2))$, for all $\vec{\epsilon}\in \{-1,+1\}^{m_f}$, there exists a homeomorphism $f_{\vec{\epsilon}}=(f_{\vec{\epsilon}}^{(1)}, f_{\vec{\epsilon}}^{(2)})$ from $B$ to $f_{\vec{\epsilon}}(B) \subseteq \R^2$ such that
\begin{itemize}
    \item $f_{\vec{\epsilon}}|_{(B \setminus l^{\vec{\epsilon}})}$ is a diffeomorphism onto its image;
    \item $f_{\vec{\epsilon}}|_{(B \setminus l^{\vec{\epsilon}})}$ is affine;
    \item $f_{\vec{\epsilon}}$ preserves $f_1$, i.e., $f_{\vec{\epsilon}}(x,y)=(x,f_{\vec{\epsilon}}^{(2)}(x,y))$;
    \item For all $i \in \{1,...,m_f\}$ and all $c\in \text{Int}(l_i)$, 
    $f_{\vec{\epsilon}}|_{(B \setminus l^{\vec{\epsilon}})}$ extends to a smooth map in the domain $\{(x,y)\in D| \ x\leq x_i\}$ and $\{(x,y)\in D| \ x\geq x_i\}$, where $D$ is an open ball around $c$. Furthermore,
    \begin{equation*}
        \lim_{\stackrel{(x,y)\rightarrow c}{x<x_i}} df_{\vec{\epsilon}}(x,y)=T^{k(c)}\lim_{\stackrel{(x,y)\rightarrow c}{ x>x_i}} df_{\vec{\epsilon}}(x,y)
    \end{equation*}
\end{itemize}
where
\begin{equation*}
    T^{k(c)}= \begin{bmatrix}
1 & 0\\
k(c) & 1
\end{bmatrix}.
\end{equation*}
Such an $f_{\vec{\epsilon}}$ is unique modulo a left composition by a transformation in $\mathcal{T}$, and its image is a representative of the polytope invariant.
\end{theorem}

The map $f_{\vec{\epsilon}}$ is sometimes referred to as {\bf straightening homeomorphism}.
In order to arrive at $f_{\vec{\epsilon}}(B)$, one cuts the set $B$ along each of the vertical lines $\mathcal{L}_i^{\epsilon_i}$ to the focus-focus values. Then the resulting image becomes simply connected, and thus there exists a global $2$-torus action on the preimage of this set. A representative of the polytope invariant can thus be seen as closure of the image of a toric momentum map.

Describing all possible choices (like for instance the signs $\epsilon_j$) by means of a group action allows to write the polytope invariant as equivalence class or orbit of this group action. Since we work later always with representatives, we refer for details to \vungoc \ \cite{vu2007moment}.

\subsection{Semitoric systems and almost toric fibrations}
Let $(M,\omega,F)$ be a semitoric system. In particular, it induces an almost toric fibration and an almost toric base. 
%We want to understand what the vanishing covector, affine monodromy, etc.\ in the underlying semitoric system are.
Let $\Delta_{\vec{\epsilon}}$ be a representative of the polytope invariant for this semitoric system. Then, by \vungoc\ \cite{vu2007moment}, the eigenline of a nodal value in $\Delta_{\vec{\epsilon}}$ is the vertical line that passes through the node. The affine monodromy can be represented, with respect to an integral affine basis of $\Delta_{\vec{\epsilon}}$, as
\begin{equation*}
    \begin{bmatrix}
1 & 0\\
k & 1
\end{bmatrix}
\end{equation*}
for some $k\in \mathbb{N}$, where $k$ is the number of focus-focus critical points in the critical fiber. %Furthermore, the vanishing covector is the dual of $\pm (1,0)$, i.e., it is $v^*:=\langle \pm (1,0), \cdot \rangle$, where $\langle \cdot , \cdot \rangle$ denotes the standard inner product.
%We wish to understand in what situation we are able to do a nodal trade on $\Delta_{\vec{\epsilon}}$.

\begin{definition}
    Let $p$ be a corner, with precisely two edges emanating from it, of a polytope in $\mathbb{R}^2$. We say that $p$ is a \textbf{Delzant corner} if there exists an integral affine transformation $A$ that maps $p$ and its adjacent edges to the corner at $(0,0)$ with adjacent edges created by $(1,0)$ and $(0,1)$, which is called the \textbf{standard Delzant corner}. This is equivalent to the fact that there exist vectors spanning the edges of $p$ that are a $\mathbb{Z}$ basis of $\mathbb{Z}^2$.
\end{definition}

%\vungoc\ \cite{vu2007moment} showed that $\Delta_{\vec{\epsilon}}$ is a rational convex polytope, i.e., it can be obtained by a finite intersection of closed half-spaces whose boundary hyperplanes admit normal vectors with integer coefficients.
\begin{remark}
\label{l.conditionsfornodaltrade}
    Let $\Delta_{\vec{\epsilon}}$ be a representative of the polytope invariant for a simple semitoric system $(M,\omega,F)$ and $R$ be an eigenray of $\Delta_{\vec{\epsilon}}$ that connects the node to a point $p$ in the boundary which is a Delzant corner. Then the assumptions of Lemma \ref{l.nodaltrade} are satisfied , i.e. the vanishing and collapsing covectors span a lattice, and we are able to do a nodal trade.
\end{remark}

\subsection{Displacing toric fibers}
In McDuff \cite{mcduff2011displacing} conditions are given to displace Lagrangian toric fibers of toric manifolds. In this section, we summarize the necessary notions and  results of McDuff \cite{mcduff2011displacing} which we will need later on.

The {\em affine distance} $d_{\text{aff}}(x,y)$ between two points $x$, $y$ on a line $L$ with rational slope is defined as the quotient of the Euclidean distance $d_E(x,y)$ and the minimum Euclidean distance from $0$ to an integral point $p \in \mathbb{Z}^n$
on the line through $0$ parallel to $L$.
Equivalently, if $\phi$ is any integral affine transformation  of $\mathbb{R}^n$ that takes $x,y$ to the first coordinate axis (the $x_1$-axis) then $d_{\text{aff}}(x,y)=d_E(\phi(x),\phi(y))$.
An affine line $L=z+\mathbb{R}v$ in $\mathbb{R}^n$ is rational if the direction vector $v$ can be taken to be a primitive integral vector in $\mathbb{Z}^n$. Given an affine rational line $L$ with primitive direction $v\in \mathbb{Z}^n$, the affine distance $d_{\text{aff}}$ between two points $x,y\in L$ is $t\in \mathbb{R}$ such that $x-y=tv\in \mathbb{R}^n$.
An affine hyperplane $A$ is called \textbf{rational} if it has a primitive integral normal vector $\eta$, i.e., if it is given by an equation of the form $\langle x,\eta \rangle=k$ where $k\in \mathbb{R}$ and $\eta$ is primitive and integral. The \textbf{affine distance} $d_{\lambda}(x,A)$ from a point $x$ to a rational affine hyperplane in the rational direction $\lambda$ is defined as
\begin{equation}
\label{eq.affinedistance}
    d_{\lambda}(x,y):=d_{\text{aff}}(x,y)
\end{equation}
where $y\in A$ lies on the ray $x+a\lambda$ with  $a\in \mathbb{R}^{+}$. If the ray does not meet $A$ we set $d_{\lambda}(x,A)=+\infty$. We say that an integral vector $\lambda$ is \textbf{integrally transverse} to $A$ if $|\langle \lambda, \eta \rangle |=1$ where $\eta$ is the normal as above. 

\begin{definition}
    Let $w$ be a point on a facet $F$, i.e., on a $(n-1)$-dimensional face of a rational polytope $\Delta$ and let $\lambda \in \mathbb{Z}^n$ be integrally transverse to $F$. The \textbf{probe} $p_{F,\lambda}(w)=:p_{\lambda}(w)$ with direction $\lambda \in \mathbb{Z}^n$ and initial point $w\in F$ is the half open line segment consisting of $w$ together with the points in $\text{int}(\Delta)$ that lie on the ray from $w$ in direction $\lambda$.
\end{definition}

In the next lemma we could use any notion of length, but the affine distance is the most natural one.

\begin{lemma}\label{l.probes}({McDuff \cite[Lemma 2.4]{mcduff2011displacing}})
    Let $\Delta$ be a smooth momentum polytope associated with a symplectic toric manifold $(M,\omega,\pi)$, i.e., $\Delta=\pi(M)$.
    Let $u\in \text{int}(\Delta)$ lie on the probe $p_{F,\lambda}(w)$. If $w$ lies in the interior of its corresponding facet and if $u$ is less than halfway measured from the boundary along $p_{\lambda}(w)$, then the fiber $\pi^{-1}(u)$ is displaceable.
\end{lemma}

\begin{definition}
    Let $\Delta$ be any rational polytope and $u\in \text{int}(\Delta)$. If there is a probe $p_{F,\lambda}(w)$ through $u$ that satisfies the conditions of Lemma \ref{l.probes} then $u$ is said to be \textbf{displaceable by the probe} $p_{F,\lambda}(w)$.
\end{definition}

\begin{remark}
    As observed in Pabiniak \cite{MR3300713}, McDuff's method of probes applies in any integrable system where the inverse image of the probe lies in a Darboux chart. We use this throughout the paper, to construct probes in representatives of the semitoric polytope invariant.
\end{remark}

\subsection{Special Hamiltonian diffeomorphisms of the disk}
\label{what}
In this subsection, we recall a result from Abreu $\&$ Borman $\&$ McDuff \cite{abreu2014displacing} that concerns Hamiltonian diffeomorphisms of the disk which we will need to displace focus-focus fibers in Section \ref{s.displacinggeneralff}.

For real numbers $0<A<B$, consider a smooth, non-increasing function $a:[0,B-A]\rightarrow [A,B]$ such that the function $s \mapsto a(s)+s$ is non-decreasing for $s\in [0,B-A]$ and
\begin{equation*}
    A\leq a(s)+s\leq B \quad \mbox{and} \quad a(B-A)=A.
\end{equation*}
%For our purposes, we pick $a(s):=B-s$. 
For $r>0$, let $\mathbb{D}(r)\subset \mathbb{R}^2$ denote the $2$-dimensional disk of area $r$ centered at the origin.

\begin{lemma}[{Abreu $\&$ Borman $\&$ McDuff, \cite[Lemma 6.4.2]{abreu2014displacing}}]
\label{l.abreumcfuddprobes}
For any function $a(s)$ as above and any $\epsilon>0$, there is a compactly supported Hamiltonian diffeomorphism $\rho: \text{Int}(\mathbb{D}(B))\rightarrow \text{Int}(\mathbb{D}(B))$ such that 
\begin{equation*}
    \rho(\mathbb{D}(s))\subset \mathbb{A}(a(s),a(s)+s+\epsilon), \qquad \forall\ \epsilon\leq s\leq B-A-\epsilon,
\end{equation*}
where $\mathbb{A}(b,c):=\mathbb{D}(c)\backslash \text{Int}(\mathbb{D}(b))$.
\end{lemma}
\begin{corollary}
\label{c.displaceablehsmaller}
If $A>\frac{B}{2}$, every disk of area less than $B-A$ is displaced by $\rho$, where $\rho$ is the Hamiltonian diffeomorphism of Lemma \ref{l.abreumcfuddprobes}.
\end{corollary}

%%%%%%%%%%%%%%%%%%%%%%%%%%%%%%%%%%%%%%%%%%%%%%%%%%%%
%%%%%%%%%% new subsection  %%%%%%%%%%%%%%%%%%%%%%%%%%%

\subsection{Symplectic quasi-states}
\label{s.states}
Symplectic quasi-states are a useful tool to find nondisplaceable subsets in a symplectic manifold, and so are the notions of superheavy, heavy and pseudoheavy sets. Entov $\&$ Polterovich \cite{entov2009rigid, polterovich2006quasi} introduced the notions of {\em heaviness} and {\em superheaviness}. Furthermore, Kawasaki \& Orita \cite{kawasaki2021existence} added the notion of {\em pseudoheaviness}. The definition of a symplectic quasi-state and pseudoheavy, heavy, superheavy subsets is based on spectral invariants in Hamiltonian Floer theory (see Oh \cite{oh2005construction}). We now recall in more detail some definitions and results from these papers that we need later on.

In what follows we require all symplectic manifolds $(M,\omega)$ to be \textbf{closed}. However we note that the theory of symplectic quasi-states could be developed for non-compact symplectic manifolds by considering compactly supported Hamiltonian functions.

\begin{definition}
    A \textbf{partial symplectic quasi-state} on a symplectic manifold $(M,\omega)$
    is a map $\zeta:C^{\infty}(M)\rightarrow \mathbb{R}$ that satisfies the following conditions:
    \begin{itemize}
        \item \textbf{Normalization:} $\zeta(a)\equiv a$ for all $a\in \mathbb{R}$.

        \item \textbf{Stability:} For all $H_1,H_2\in C^{\infty}(M)$:
        \begin{align*}
            \min_M(H_1-H_2)\leq \zeta(H_1)-\zeta(H_2)\leq \max_M(H_1-H_2).
        \end{align*}

        \item \textbf{Semi-Homogenity:} $\zeta(sH)=s\zeta(H)$ for all $H\in C^{\infty}(M)$ and all $s>0$;
        \item \textbf{Hamiltonian invariance:} $\zeta(H\circ \phi)=\zeta(H)$ for all $H\in C^{\infty}(M)$ and all $\phi \in \text{Ham}(M,\omega)$.
        \item \textbf{Vanishing:} $\zeta(H)=0$ for all $H\in C^{\infty}(M)$ with support displaceable from itself.
        \item \textbf{Quasi-subaditivity:} $\zeta(H_1+H_2)\leq \zeta(H_1)+\zeta(H_2)$ for all $H_1,H_2\in C^{\infty}(M)$ satisfying $\{H_1,H_2\}=0$, where $\{\cdot,\cdot \}$ is the Poisson bracket induced by the symplectic form.
    \end{itemize}	
\end{definition}

\begin{definition}
    \textbf{Symplectic quasi-states} are \textbf{partial symplectic quasi-states} that satisfy the following \textbf{quasi-linearity} assumption: if $\{H,K\}=0$ then $\zeta(H+aK)=\zeta(H)+a\zeta(K)$ for any $a\in \mathbb{R}$.
\end{definition}

\begin{definition}
    A map $\zeta:C^{\infty}(M,\omega)\rightarrow \mathbb{R}$ satisfies the so-called \textbf{PB-inequality} if there exists $C>0$ such that
    \begin{equation*}
        |\zeta(H+K)-\zeta(H)-\zeta(K)|\leq C\sqrt{||\{H,K\}||} \quad \forall\ H, K \in C^{\infty}(M,\omega).
    \end{equation*}
\end{definition}

\begin{definition}
    Let $\zeta:C^{\infty}(M)\rightarrow \mathbb{R}$ be a $\textbf{partial symplectic quasi-state}$ on $(M,\omega)$.
    A closed subset $X$ of $M$ is $\zeta\textbf{-heavy}$ (resp. $\zeta\textbf{-superheavy}$) if
    \begin{align*}
        \zeta(H)\geq \inf_X H \quad (resp. \hspace{2mm} \zeta(H)\leq \sup_XH)
    \end{align*}
for all $H\in C^{\infty}(M)$.
    A closed subset $X$ is said to be $\zeta \textbf{-pseudoheavy}$ if for all open neighborhood $U$ of $X$ there exists a function $F\in C^{\infty}(U)$ such that $\zeta(F)>0$.
\end{definition}

These sets have important properties and relations between them:

\begin{theorem}
(Kawasaki $\&$ Orita \cite[Proposition 1.8]{kawasaki2021existence}, Entov $\&$ Polterovich \cite[Theorem $1.4$]{entov2009rigid})
\label{t.properties}
Let $\zeta:C^{\infty}(M)\rightarrow \mathbb{R}$ be a partial symplectic quasi-state on $(M,\omega)$:
\begin{enumerate}
    \item Every $\zeta$-\textbf{superheavy} subset is $\zeta$\textbf{-heavy}.
    \item Every $\zeta$-\textbf{heavy} subset is $\zeta$\textbf{-pseudoheavy}.
    \item Every $\zeta$-\textbf{pseudoheavy} subset is nondisplaceable from itself.
    \item Every $\zeta$-\textbf{pseudoheavy} subset is nondisplaceable from every $\zeta$\textbf{-superheavy} subset.
\end{enumerate}
\end{theorem}

One important result is that there always exists a \textbf{pseudoheavy} fiber for a partial symplectic quasi-state on an integrable system $(M,\omega,F):$

\begin{theorem}({Kawasaki $\&$ Orita, \cite[Theorem 1.7]{kawasaki2021existence}})
\label{t.pseudo}
    Let $(M,\omega,F)$ be an integrable system. 
    Let $\zeta:C^{\infty}(M)\rightarrow \mathbb{R}$ be a partial symplectic quasi-state on $(M,\omega)$. Then there exists $y_0\in F(M)$ such that $F^{-1}(y_0)$ is $\zeta$-pseudoheavy.
\end{theorem}

\begin{remark}
Entov \& Polterovich \cite{polterovich2006quasi} showed the existence of a partial symplectic quasi-state for $2n$-dimensional symplectic manifolds which are rational and strongly semi-positive. Due to the work of Oh \cite{oh2009floer} and Usher \cite{dualityusher,spectralusher} the conditions can be dropped, see Entov \cite{entovclarification} for more details. Note that all symplectic 4-manifolds are strongly semi-positive.
\end{remark}

The existence of partial symplectic quasi-states on symplectic 4-manifolds thus implies:

\begin{theorem}
\label{t.existsnondisplaceable}
    (Entov $\&$ Polterovich, \cite[Theorem 2.1]{polterovich2006quasi}, Usher \cite{dualityusher,spectralusher}, Oh \cite{oh2009floer}) Any integrable system $(M,\omega,F)$ on a symplectic $4$-manifold $(M,\omega)$ has a nondisplaceable fiber.
\end{theorem}

\begin{definition}
Let $(M,\omega,F)$ be an integrable system and $p\in F(M)$. A fiber $F^{-1}(p)$ is called a \textbf{stem} if all other fibers of $F$ are displaceable.
\end{definition}

\begin{theorem}
\label{t.stem}
(Polterovich $\&$ Rosen, \cite[Corollary 6.1.6]{polterovich2014function}) Any stem is nondisplaceable by the group of symplectomorphisms of $(M,\omega)$.
\end{theorem}

Borman \cite{borman2013quasi} showed that symplectic quasi-states behave well under symplectic reduction on a superheavy subset: let $(W,\omega)$ be a $2n$-dimensional closed symplectic manifold equipped with a smooth map $\Phi=(\Phi_1,\cdot\cdot\cdot,\Phi_k):W\rightarrow \mathbb{R}^k$ and a regular level set $Z: =\Phi^{-1}(0)$. Suppose that all $\Phi_i$ mutually Poisson commute at each point in $Z$ and that $\Phi$ induces a free Hamiltonian $\mathbb{T}^k$-action on $Z$. Let $(M^{\text{red},0}:=Z/\mathbb{T}^k,\omega^{\text{red},0})$ be the reduced space when performing symplectic reduction at $0$. Denote by $\pi:Z\rightarrow M^{\text{red},0}$ the quotient map.

\begin{theorem}(Borman, \cite[Theorem 1.1]{borman2013quasi})
\label{t.reductionquasistates}
Let $(W,\omega)$ be a $2n$-dimensional closed symplectic manifold and $Z$ as above.
    If $\zeta:C^{\infty}(W,\omega)\rightarrow \mathbb{R}$ is a {symplectic quasi-state} satisfying the {PB-inequality} and $Z$ is {superheavy} for $\zeta$, then $\zeta$ induces a {symplectic quasi-state}
    \begin{equation*}
        \zeta^{\text{red},0}:C^{\infty}(M^{\text{red},0},\omega^{\text{red},0})\rightarrow \mathbb{R}
    \end{equation*}
    satisfying the {PB-inequality}. If $Y\subset Z$ is superheavy for $\zeta$ then $\pi(Y)$ is superheavy for $\zeta^{\text{red},0}.$
\end{theorem}

\section{Displacing focus-focus fibers}
\label{s.displacingfocusofocusfibersiwthgeneralized}
\subsection{Generalization of the method of probes}
\label{s.generalization}
In this subsection, we generalize the method of probes, which we will apply in subsequent sections to displace focus-focus fibers. The idea behind our generalization of the method of probes is to embed a family of probes in a toric fibration, obtaining an embedding of the product of two disks. 

\begin{lemma}
\label{l.generalizedprobes}
Let $F:(M,\omega)\rightarrow B$ be a toric fibration. Suppose that there exists an affine embedding of $([0,R]\times [0,h],\mathcal{A}_0)$ into $(B,\mathcal{A})$ where $\mathcal{A}$ is the integral affine structure induced by $F$ on $B$, $R>0$ and $0<h<\frac{R}{2}$. Then the set $F^{-1}(\text{Im}([0,h]\times [0,h]))\subset M$, where $\text{Im}([0,h]\times [0,h]))$ is the image by the affine embedding of $([0,h]\times [0,h],\mathcal{A}_0)$, is \textbf{displaceable}.

\begin{proof}
Consider the vectors % Let us denote the integral normal vectors by
$v_1=(1,0)$ and $v_2=(0,1)$, and notice that $\langle v_1,v_1\rangle = \langle v_2,v_2\rangle=1$ and $\langle v_1,v_2\rangle =0$. 
Since we are in a toric fibration we have well-defined action-angle coordinates $(x,\theta)\in B\times \mathbb{T}^2$. Using these action-angle coordinates and the assumptions in the statement, possibly after applying an integral affine transformation, we have an embedding
   \begin{equation*}
        \Psi: \mathbb{D}(R)\times \mathbb{D}(h)\rightarrow (M,\omega), \qquad 
        (s,\theta_1,t,\theta_2)\mapsto (sv_1+tv_2,\theta_1v_1+\theta_2v_2)\in B\times \mathbb{T}^2
    \end{equation*}
where $s$ and $t$ parametrize the radii of the disks.
Let $\pi_i$ be the projections onto the disks for $i=1,2$. Then $\Psi^*\omega= \pi_1^*(\omega_0)+\pi_2^*(\omega_0)$.

Since $h<\frac{R}{2}$ we can use the results of Section \ref{what}, more specifically Corollary \ref{c.displaceablehsmaller}, to obtain a compactly supported Hamiltonian $G_h:\mathbb{D}(R)\rightarrow \mathbb{R}$ that displaces every disk of area less than or equal to $h$. This Hamiltonian allows us to define the compactly supported Hamiltonian 
\begin{equation*}
    G: \mathbb{D}(R)\times \mathbb{D}(h) \quad \rightarrow \quad \mathbb{R}, \qquad  G:=G_h\circ \pi_1.
\end{equation*}
Since this Hamiltonian does not depend on $\mathbb{D}(h)$ and $\Psi^*\omega=\pi_1^*(\omega_0)+ \pi_2^*(\omega_0)$, the flow of $G\circ \Psi^{-1}$ displaces the set $F^{-1}([0,h]\times [0,h])$ inside $\text{Im}(\Psi)$. Since $H$ is compactly supported we can extend it to the whole manifold $(M,\omega)$.
\end{proof}
\end{lemma}

The above proof can be adapted to obtain the following corollary:
\begin{corollary}
    Let $F:(M,\omega)\rightarrow B$ be a toric fibration. Suppose that there exists an affine embedding of $([0,R]\times [0,a],\mathcal{A}_0)$ into $(B,\mathcal{A})$ where $\mathcal{A}$ is the integral affine structure induced by $F$ on $B$, $R,a>0$ and  $0<h<\frac{R}{2}$. Then the set $F^{-1}(\text{Im}([0,h]\times [0,a]))\subset M$ is \textbf{displaceable}.
\end{corollary}

%The method mentioned above can be applied in the case of a semitoric system in order to displace a focus-focus fiber. Suppose that we are able to obtain one representative of the polytope invariant such that the introduced vertex is equivalent to the standard Delzant corner. Then we can apply an integral affine transformation to be in a polytope where our new vertex corresponds to the standard Delzant one. Then if the eigenray sits inside a square of length less than half the affine distance of the corresponding edges we may displace the focus-focus fiber, after applying a nodal trade. Notice that the fact if the eigenline is contained in such a square is related to the \textbf{height invariant} of the semitoric system. We make this discussion concrete in Section \ref{s.displacinggeneralff}.

\subsection{Nondisplaceability of focus-focus fibers with higher multiplicity}
\label{s.nondisplaceabilitybigk}
Let $c$ be a focus-focus value of a semitoric system $F$ on a $4$-dimensional symplectic manifold $(M,\omega)$, and let $k$ be the number of focus-focus points on the fiber $F^{-1}(c)$, often called the multiplicity of the focus-focus fiber $F^{-1}(c)$. Intuitively, a fiber with $k$ focus-focus points is often seen as a torus with $k$ pinches. 
Alternatively the fiber $F^{-1}(c)$ can be seen as a closed chain of $k$ spheres joint at the poles, see Zung \cite{zungarnold} and \vungoc\ \cite[Section $6$]{vu2000bohr}. In particular, these are Lagrangian spheres. If $k=1$ the sphere is immersed, and if $k>1$ the spheres are embedded.

\begin{theorem}
\label{p.lagrangiansphere}
    An embedded Lagrangian sphere in a symplectic $4$-manifold is \textbf{nondisplaceable}.
\end{theorem}
\begin{proof}
   Recall the notion of intersection number and Euler characteristic (for a reference see Hirsch \cite[Section $5.2$]{hirsch2012differential}). 
    
    Consider an embedded Lagrangian sphere $i:L\rightarrow (M,\omega)$ where $(M,\omega)$ is a $4$-dimensional symplectic manifold. Using Weinstein's Lagrangian neighborhood theorem, we identify a neighborhood of $i(L)$ with the cotangent bundle of $\mathbb{S}^2$. Therefore the intersection number $L\cdot L$ is such that $|L\cdot L|= \chi(\mathbb{S}^2)\neq 0$. Now we prove that $i(L)$ is not topologically displaceable by a smooth map $g:L\rightarrow M$ homotopic to the inclusion $i$. Let us argue by contradiction. Assume that $g(i(L))\cap i(L)=\emptyset$ and $g\simeq i$. In particular, $g$ and $i$ are transverse, and since $g(i(L))\cap i(L)=\emptyset$ we have $L\cdot L= \#(g,L)=0$, where $\#(g,L)$ is the intersection number of the map $g$ with $L$. Therefore we obtain a contradiction. Since $L$ is not topologically displaceable, in particular it is not displaceable by a Hamiltonian diffeomorphism.
\end{proof}
\begin{corollary}
\label{c.highernondisplaceable}
    A focus-focus fiber $F^{-1}(c)$ is \textbf{nondisplaceable} when its multiplicity is greater than or equal to $2$.
\end{corollary}

Corollary \ref{c.highernondisplaceable} settles the (non)displaceability questions for focus-focus fibers of semitoric systems, with multiplicity higher than one. In Section \ref{s.octagon} we study a family of semitoric systems which presents multiple focus-focus fibers of multiplicity two.
%Therefore studying (non)displaceability of focus-focus fibers is reduced to the case $k=1$, which is the generic case, see Smirnov \cite{smirnov2013focus}. Furthermore, In Section \ref{s.octagon} we study a family of semitoric systems which presents multiple focus-focus fibers of multiplicity $2$.

%There are also explicit examples for higher multiplicity. For the case $k=2$ see for instance De Meulenaere \& Hohloch \cite{de2021family} and Section \ref{s.octagon}.

\subsection{Displacing focus-focus fibers}
\label{s.displacinggeneralff}
In this subsection we explain how to displace focus-focus fibers in certain semitoric systems using the methods of Subsection \ref{s.generalization}.

Let $(M,\omega,F)$ be a semitoric system on a $4$-dimensional symplectic manifold $(M,\omega)$ and to reduce notation assume that we have a single focus-focus value $c$. Due to Subsection \ref{s.nondisplaceabilitybigk} we focus on the case of multiplicity $1$.

Let $\Delta_{\vec{\epsilon}}$ be a representative of the polytope invariant for the semitoric system. Without loss of generality we henceforth assume that $\vec{\epsilon}=-1$.

Let $h$ be the height invariant for the focus-focus value and $v_1,v_2$ the edges that intersect at the corner introduced by the focus-focus value in the polytope invariant, as sketched in Figure \ref{f.displacingfffiber}.
\begin{figure}
\begin{center}
    \includegraphics{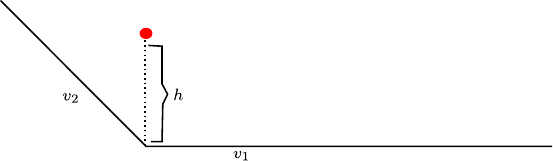}
    \caption{Part of the representative of the polytope invariant for $\vec{\epsilon}=-1$. The red dot represent the focus-focus value.}
    \label{f.displacingfffiber}
\end{center}
\end{figure}

Now assume that the corner introduced by the focus-focus fiber is a Delzant corner, and let $A$ be the integral affine transformation that takes it into the standard Delzant corner. By Remark \ref{l.conditionsfornodaltrade} we can apply a nodal trade (which is defined in Lemma \ref{l.nodaltrade}). Recall the affine distance from Equation \eqref{eq.affinedistance} and denote the affine length by $\text{aff}$. Without loss of generality assume that $\text{aff}(v_1)\geq \text{aff}(v_2)$. Suppose that $h<\text{aff}(v_2)$, and that there exists $R>0$ such that $R<\text{aff}(v_1)$ and $0<2h<R$. Furthermore, suppose that  the rectangle $[0,R]\times [0,h]$ starting at $(0,0)$ sits inside $A(\Delta_{\vec{\epsilon}})$, as is sketched in Figure \ref{f.ffonthestandarddelzantcorner}.
Then we obtain the following result:
\begin{theorem}
\label{l.displaceaingff}
The focus-focus fiber of the system $(M,\omega,F)$ described above is \textbf{displaceable}.
\begin{proof}
Recall that for the standard Delzant corner the corner is placed at the origin. 
After applying a nodal trade, as defined in Lemma \ref{l.nodaltrade}, the semitoric system with focus-focus fiber $F^{-1}(c)$ becomes a toric fibration $\pi:(M,\omega)\rightarrow \mathbb{R}^2$. Hereby the set $F^{-1}(c)$ is contained in $\pi^{-1}([0,h]\times [0,h])$. Since $h<\frac{R}{2}$, by Lemma \ref{l.generalizedprobes}, the set $\pi^{-1}([0,h]\times [0,h])$ is displaceable, in particular the focus-focus fiber is displaceable.
\end{proof}
\end{theorem}

\begin{figure}
\begin{center}
    \includegraphics{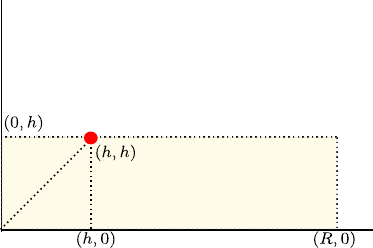}
    \captionof{figure}{Rectangle $[0,R]\times [0,h]$ sitting inside $A(\Delta_{\vec{\epsilon}})$. The red dot represents the focus-focus value.}
    \label{f.ffonthestandarddelzantcorner}
\end{center}
\end{figure}

%%%%%%%%%%%%%%%%%%%%%%%%%%%%%%%%%%%%%%%%%%%%%%%%%%%%
%%%%%%%%%% new section  %%%%%%%%%%%%%%%%%%%%%%%%%%%

\section{(Non)displaceability in three explicit systems}
\label{s.examples}
In this section we consider (non)displaceability of fibers of the coupled spin-oscillator, Subsection \ref{s.spin}, the coupled angular momenta, Subsection \ref{s.coupledangularmomenta}, and a semitoric system on the symplectic manifold given by the octagon, Subsection \ref{s.octagon}.
\subsection{(Non)displaceability within the coupled spin-oscillator}
\label{s.spin}
In this section, we study displaceability properties of fibers of the coupled spin-oscillator, a non-compact integrable system.

Let $\rho_1,\rho_2>0$ be positive constants. Consider the product manifold $M=\mathbb{S}^2\times \mathbb{R}^2$ with symplectic form $\omega=\rho_1\omega_{\mathbb{S}^2}\bigoplus \rho_2\omega_{\mathbb{R}^2}$, where $\omega_{\mathbb{S}^2}$ and $\omega_{\mathbb{R}^2}$ are the standard symplectic structures on the unit sphere and the Euclidean space respectively. Let $(x,y,z)$ be Cartesian coordinates on the unit sphere $S^{2}\subset \mathbb{R}^3$ and $(u,v)$ Cartesian coordinates on the plane $\mathbb{R}^2$.

\begin{definition}
A \textbf{coupled spin-oscillator} is a $4$-dimensional integrable system $(M,\omega,(L,H))$, where the momentum map $F=(L,H):M\rightarrow \mathbb{R}^2$ is given by 
\begin{equation}
L(x,y,z,u,v):=\rho_1z+\rho_2\frac{u^2+v^2}{2}, \quad H(x,y,z,u,v):=\frac{xu+yv}{2}.
\end{equation}
\end{definition}
Coupled spin-oscillators are semitoric systems, for more details and background see Pelayo \& \vungoc \ \cite{pelayo2012hamiltonian}. The system has exactly one focus-focus singularity at the point $m:=(0,0,1,0,0)$. In order to understand the displaceability properties of the fibers of these systems we look at a representative of the polytope invariant, see Figure \ref{f.CSO}. These were computed in Pelayo \& \vungoc\  \cite{pelayo2012hamiltonian} and Alonso \cite[Section $5.1$]{alonso2019symplecticthesis}.

%In Theorem \ref{t.CSO} we show that all fibers of this system are displaceable. This is not surprising due to the fact that fibers of this integrable system are compact and there is the $\mathbb{R}^2$ component. Nevertheless, we provide a proof by looking at a representative of the polytope invariant, Figure \ref{f.CSO}, and using the methods of Section \ref{s.displacinggeneralff}. 

\begin{figure}[h]
\begin{center}
    \includegraphics[]{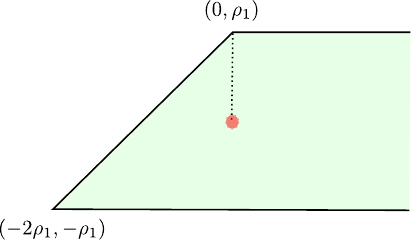}
    \caption{ Representative of the polytope invariant of the coupled-spin-oscillator with $\vec{\epsilon}=+1$. The red dot represents the focus-focus value and is placed at $(0,0)$. Notice that the polytope is \textbf{unbounded}.}
    \label{f.CSO}
\end{center}
\end{figure}

\begin{theorem}
\label{t.CSO}
All of the fibers of a coupled spin-oscillator are \textbf{displaceable}. 
\end{theorem}
\begin{proof}
There are two cases:
\begin{itemize}
\item 
First consider the fibers outside of the segment of the eigenline through the focus-focus value in Figure \ref{f.CSO}. By Lemma \ref{l.probes} all of these fibers are \textbf{displaceable} by a probe with direction $(1,0)$ starting at the facet $y=x+\rho_1$, possibly after getting rid of the segment of the eigenline of the focus-focus value by applying a small enough nodal trade for each fiber we are working with, see Lemma \ref{l.nodaltrade}. 

\item To displace the focus-focus fiber and the fibers lying on the segment of the eigenline of Figure \ref{f.CSO} we use the results of Section \ref{s.displacinggeneralff}, Lemma \ref{l.generalizedprobes} and Theorem \ref{l.displaceaingff}, since one of the edges has infinite affine length.
\end{itemize}
\end{proof}

Intuitively, it is the non-compactness of the $\mathbb{R}^2$ factor that allows us to displace every fiber.

 \subsection{(Non)displaceability within the coupled angular momenta}
\label{s.coupledangularmomenta}

In this section, we study displaceability properties of fibers of the following systems:

\begin{definition}
\label{d.couplemomenta}
Let $R_1,R_2\in \mathbb{R}^{>0}$ and $t\in [0,1]$ a parameter. Consider the product manifold $M:=\mathbb{S}^2\times \mathbb{S}^2$ with symplectic form $\omega:=-(R_1\omega_{\mathbb{S}^2}\bigoplus R_2\omega_{\mathbb{S}^2})$ where $\omega_{\mathbb{S}^2}$ is the standard symplectic form on $\mathbb{S}^2$.
Let $(x_1,y_1,z_1,x_2,y_2,z_2)$ be Cartesian coordinates on $M$ induced from the ambient $\R^3 \times \R^3$. The \textbf{coupled angular momenta system} is a family of $4$-dimensional completely integrable systems $(M,\omega,F_t:= (L,H_t))$,
where the smooth functions $L,H_t:M\rightarrow \mathbb{R}$ are given by
\begin{equation*}
\begin{cases}
L(x_1,y_1,z_1,x_2,y_2,z_2):=R_1z_1+R_2z_2, \\
H_t(x_1,y_1,z_1,x_2,y_2,z_2):=(1-t)z_1+t(x_1x_2+y_1y_2+z_1z_2)
\end{cases}
\end{equation*}
and $R_1<R_2$. The case $R_1>R_2$ is called \textbf{reverse coupled angular momenta} and the case $R_1=R_2$ is called the \textbf{Kepler problem}. 
\end{definition}

The system has four fixed points located at $(0,0,\pm 1,0 ,0, \pm 1)$. Three are of elliptic-elliptic type, namely $(0,0,\pm 1,0 ,0, 1)$, but $(0,0,1,0,0,-1)$ is of focus-focus type for $t\in\ ]t^{-},t^{+}[$ and degenerate for $t\in \{t^{-},t{+}\}$, where
\begin{equation*}
    t^{\pm}:=\frac{R_2}{2R_2+R_1\mp 2\sqrt{R_1R_2}}.
\end{equation*}
Note that at the parameter $t=\frac{1}{2}$ there is always a focus-focus singularity. For more details and proofs of these results, see for instance Alonso \cite[Section 6]{alonso2019symplecticthesis} and Alonso \& Dullin \& Hohloch \cite{alonso2019symplectic}.

%%%%%%%%%%%%%%%%%%%%%%%%%%%%%%%%%%%%%%%%%%%%%%%%%%%%
%%%%%%%%%% new subsection  %%%%%%%%%%%%%%%%%%%%%%%%%%%

\subsubsection{(Non)displaceability of the fibers}
We now study the (non)displaceability of the fibers of the systems $(M,\omega,F_t)$ in Definition \ref{d.couplemomenta}. Let us consider the map $\Psi:M\rightarrow M$ given by
\begin{equation}
\label{eq.HamDiffeoPsi}
    \Psi(x_1,y_1,z_1,x_2,y_2,z_2):=(-x_1,y_1,-z_1,x_2,-y_2,-z_2).
\end{equation}
Notice that $\Psi$ is a Hamiltonian diffeomorphism since it is the time-$1$ map of a suitable normalization of the Hamiltonian $G: M \to \R$, $G	(x_1,y_1,z_1,x_2,y_2,z_2):=R_1y_1+R_2x_2$.
Using $\Psi$, we obtain the following result:

\begin{proposition}
\label{p.general}
Consider $(a,b)\in F_t(M)$ with $a\neq 0$.
Then the fiber $(F_t)^{-1}(a,b)$ is displaceable for $t\in [0,1]$.
\end{proposition}
\begin{proof}
For $p:=(x_1,y_1,z_1,x_2,y_2,z_2)\in M$ we have $L(p)=R_1z_1+R_2z_2$ and
\begin{equation*}
    L(\Psi(p))=-(R_1z_1+R_2z_2)=-L(p).
\end{equation*}
Therefore, whenever $a\neq 0$, we deduce
\begin{equation*}
    F_t^{-1}(a,b)\cap \Psi(F_t^{-1}(a,b))=\emptyset.
\end{equation*}
\end{proof}
\begin{corollary}
Let $R_1\neq R_2$ and $t\in \ ]t^{-},t^{+}[$, the focus-focus fiber of $(M,\omega,F_t)$ is displaceable. 
\end{corollary}
\begin{proof}
The focus-focus value is $(R_1-R_2,1-2t)$. Hence the result follows from Proposition \ref{p.general}.
\end{proof}
It remains to investigate the (non)displaceability properties of $F^{-1}(a,b)$ where $(a,b)\in F_t(M)$ with $a=0$. For the rest of this section assume $R_2>R_1$. The case $R_1>R_2$ is analogous. The case $R_2=R_1$ will be dealt with separately in Section \ref{s.keplerproblem}.

We consider the system for $t^{-}<t<t^{+}$, i.e., there are three elliptic-elliptic singularities and one singularity of focus-focus type. For the following statement, we will make use of the polytope invariant (cf.\ Section \ref{s.semitoric}) of the coupled angular momenta system. One of the representatives of the polytope invariant, after applying a horizontal translation $(x,y)\mapsto (x+(R_2-R_1),y)$, is drawn in Figure \ref{f.1}.

\begin{figure}[h]
\begin{center}
    \includegraphics[]{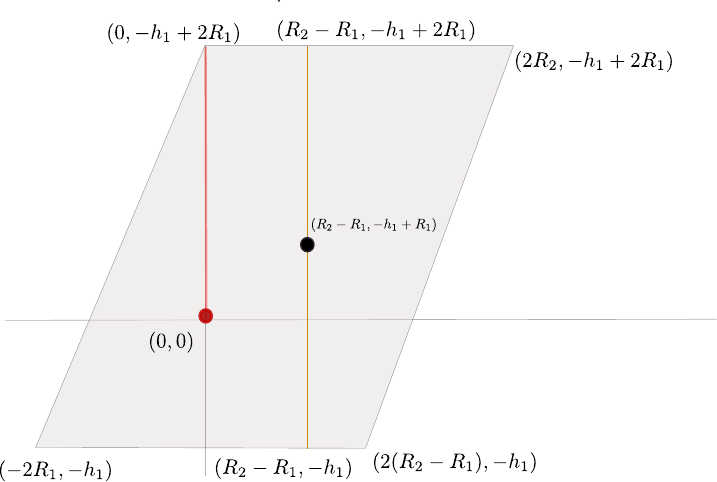}
    \caption{ A horizontal translation of a representative of the polytope invariant for $\vec{\epsilon}=1$ of the coupled angular momenta system where $t^{-}<t<t^{+}$. The red segment stands for the cut from above to the focus-focus value along the eigenline. The orange line represents the vertical line $x=R_2-R_1$. The value $h_1$ is the height invariant of the system. The nondisplaceable fiber is the preimage of the black dot.}
    \label{f.1}
\end{center}
\end{figure}

\begin{theorem}
\label{p.fibersofcoupledangularmomenta}
Let $R_1  \neq R_2$ and $t^{-}<t<t^{+}$. Then
\begin{itemize}
\item there only exists one nondisplaceable fiber in the coupled angular momenta system $F_t=(L,H_t)$, namely $(F_t^{-1}\circ f^{-1})(R_2-R_1,-h_1+R_1)$, where $h_1$ is the height invariant of the system and $f$ is the horizontal translation of a representative of the polytope invariant (cf.\ Theorem \ref{th.straigheningHomeo}), as sketched in Figure \ref{f.1}. Therefore the nondisplaceable fiber is a stem and nondisplaceable by symplectomorphisms. This fiber is independent of the choice of representative of the polytope invariant. 
\item All other fibers are displaceable. 
\item In particular, the focus-focus fiber is \textbf{displaceable}.
\end{itemize}
\end{theorem}
\begin{proof} 
Note that the map $f$, defined as the composition of a horizontal translation with the straightening homeomorphism $f_{\vec{\epsilon}}$ of a representative of the polytope invariant, maps the vertical line $x=0$ in $F_t(M)$ to the vertical line $x=R_2-R_1$ in $f(F_t(M))$, see Alonso \& Dullin \& Hohloch \cite{alonso2019symplectic}.

We can apply Lemma \ref{l.probes} as long as the probe does not intersect the segment of the eigenline, since outside this segment the semitoric system is a toric fibration. Hence, in this polytope, by using a probe with direction $(0,\pm 1)$, every fiber of the form $F_t^{-1}\circ f^{-1}(R_2-R_1,b)$ with $b\neq -h_1+R_1$ is displaceable.
Furthermore, using Theorem \ref{t.existsnondisplaceable} and Proposition \ref{p.general}, we conclude that the fiber  $F_t^{-1}\circ f^{-1}(R_2-R_1,-h_1+R_1)$ is nondisplaceable. In particular it is a stem, and by Theorem \ref{t.stem} nondisplaceable by symplectomorphisms.
\end{proof}

%%%%%%%%%%%%%%%%%%%%%%%%%%%%%%%%%%%%%%%%%%%%%%%%%%%%
%%%%%%%%%% new subsection  %%%%%%%%%%%%%%%%%%%%%%%%%%%

\subsubsection{(Non)displaceability in the Kepler problem}
\label{s.keplerproblem}

Now we consider the case $R_1=R_2=:R$ which is known as the Kepler problem. Filling $R$ into the formula of $t^\pm$, we find $t^{-}=\frac{1}{5}$ and $t^{+}=1$.

We will study the system for $\frac{1}{5}<t<1$, where the system has precisely one focus-focus point.
By Proposition \ref{p.general}, every fiber of $F_t$ over $(a,b)$ with $a\neq 0$ is displaceable. Let us now consider a fiber $F_t^{-1}(0,b)$ with $b\in \mathbb{R}$.

\begin{proposition}
\label{p.displacingeasykepler}
Consider the Kepler problem with $\frac{1}{5}<t<1$. Then the fiber $F_t^{-1}(0,b)$ is displaceable for $b\in ([-1,-t[ \ \cup\ \mathbb{R}^{>0})\cap F_t(M)$.
In particular, if $\frac{1}{5}<t<\frac{1}{2}$, the focus-focus fiber $F_t^{-1}(0,1-2t)$ is displaceable.
\end{proposition}

\begin{proof}
Recall the Hamiltonian diffeomorphism $\Psi$ from Equation \eqref{eq.HamDiffeoPsi}.
Consider $p:=(x_1,y_1,z_1,x_2,y_2,z_2)\in M$ and set $b:=H_t(p)=(1-t)z_1+t(x_1x_2+y_1y_2+z_1z_2)$. Then
\begin{align*}
    H_t(\Psi(p)) & = -(1-t)z_1+t(-x_1x_2-y_1y_2+z_1z_2)
    \\
    & = -(1-t)z_1-t(x_1x_2+y_1y_2+z_1z_2)+2tz_1z_2
    \\
    &= -H_t(p)+2tz_1z_2
    \\
    &= -b + 2tz_1z_2.
\end{align*}
On $L^{-1}(0)$, we have $z_1=-z_2$. Setting $z:=z_1= -z_2$, we find that $H_t(\Psi(p))=-b-2tz^2$ on $L^{-1}(0)$.
Therefore we have 
\begin{equation*}
    F_t(\Psi(F_t^{-1}(0,b)))\subset \{0\}\times [-b-2t,-b].
\end{equation*}
For $b<-t$, we have $-b-2t>-t$ and hence
\begin{equation*}
    \Psi(F_t^{-1}(0,b))\cap F_t^{-1}(0,b) =\emptyset.
\end{equation*}
Moreover, for $b\in \mathbb{R}^{>0}$, we have $[-b-2t,-b]\cap \{b\}=\emptyset$,  hence the desired fibers are displaceable by $\Psi$.
Note that the focus-focus value is $(0,1-2t) \in \R^2$. Hence, for $t<\frac{1}{2}$, the focus-focus fiber $F_t^{-1}(0,1-2t)$ is displaceable.
\end{proof}

Let us now inquire about the (non)displaceability of the focus-focus fiber for $t\geq \frac{1}{2}$. First we need to establish some facts about the height invariant. From now on $h:=h_t$ is the height invariant associated with the Kepler problem $(M,\omega,F_t)$. In Alonso \& Dullin \& Hohloch \cite{alonso2019symplectic} the authors compute the height invariant of the Kepler problem and show that it is given by 
\begin{equation*}
 h=R\tilde{h}
\end{equation*}
where
\begin{equation*}
\tilde{h}:=2-\frac{2}{\pi}\left(2\arctan(e^{-v})-sech(v))\right),\quad v:=\text{arctanh} \left(\frac{1-3t}{2t}\right).
\end{equation*}
In order to understand the displaceability properties we need to understand how $h=h_t$ varies with $t$. Figure \ref{f.plotoftildeh} shows a plot of $\tilde{h}$.
\begin{figure}
\begin{center}
    \includegraphics[width=8cm]{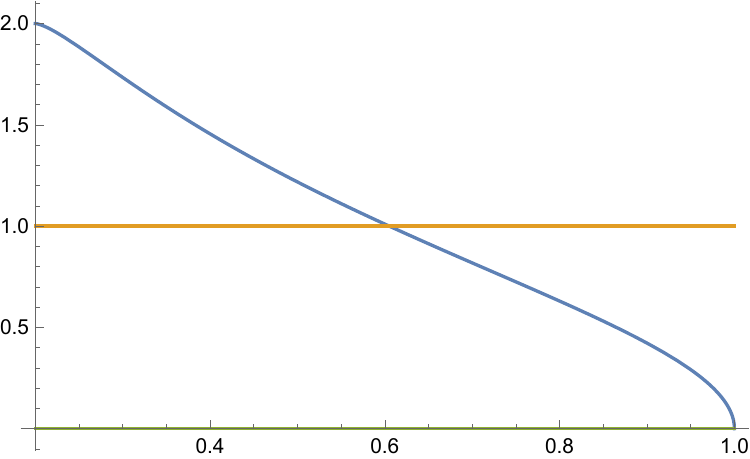}
    \caption{Plot of $\tilde{h}$ for $t\in \ ]t^{-},t^{+}[$ in blue. Plot of the constant line equal to $1$ in orange. Plot of the constant line equal to $0$ in green.}
    \label{f.plotoftildeh}
\end{center}
\end{figure}
Let $t^{-}<t_0<t^{+}$ be the value determined by $\tilde{h}(t_0)=1$. Define the maps $T_{0,h-2R}: \R^2 \to \R^2$ and $A: \R^2 \to \R^2$ via
\begin{equation}
\label{eq.mapsfordelzantcorner}
 T_{0,h-2R}(x,y):=(x,y+(h-2R))
 \quad \mbox{and} \quad
     A := \begin{bmatrix}
  0 & -1\\
  1 & -1
    \end{bmatrix} \in GL(n,\mathbb{Z}).
\end{equation}

\begin{theorem}
\label{pr.keplerOneNondispl}
    For  $ \frac{1}{5}<t<t_0$, the Kepler problem $(M,\omega,F_t)$ has a unique nondisplaceable fiber. In particular it is a stem and nondisplaceable by symplectomorphisms.
    It is given by $F_t^{-1}\circ f_{\vec{\epsilon}}^{-1}\circ T_{0,h-2R}^{-1}\circ A^{-1}(R,R)$ where $A$ and $T_{0,h-2R}$ are the maps defined in \eqref{eq.mapsfordelzantcorner} and $f_{\vec{\epsilon}}$ is the straightening homeomorphism associated with
    the representative of the polytope invariant given in Figure \ref{f.polytopeinvaraintofkeplerproblem}. In particular, the focus-focus fiber is \textbf{displaceable}.
\end{theorem}
\begin{proof}
Consider the representative with $\vec{\epsilon}=1$ of the polytope invariant plotted in Figure \ref{f.polytopeinvaraintofkeplerproblem}. %see Alonso \& Dullin \& Hohloch \cite{alonso2019symplectic}. 
Furthermore, notice that applying the transformation $A\circ T_{0,h-2R}$ to the polytope invariant in Figure \ref{f.polytopeinvaraintofkeplerproblem} we obtain Figure \ref{f.affinetrasnformationappliedtopolytopeinvariant}. Furthermore the set $J^{-1}(0)$ is mapped to the diagonal in Figure \ref{f.affinetrasnformationappliedtopolytopeinvariant}.
\begin{figure}
\begin{center}
    \includegraphics[width=0.6\textwidth]{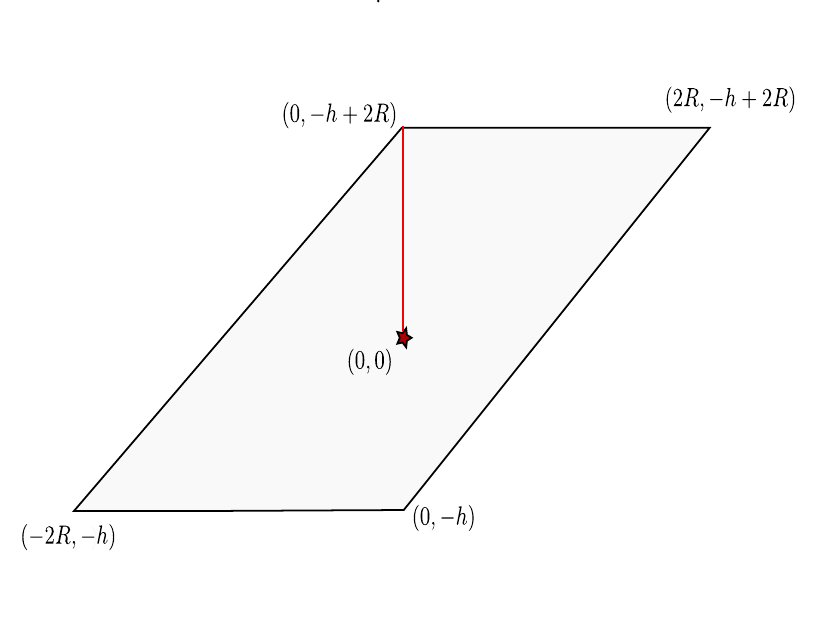}
    \caption{Representative of the polytope invariant with $\vec{\epsilon}=1$ for the Kepler problem for $\frac{1}{5}<t<1$. The cut along the eigenline is sketched in red.}
    \label{f.polytopeinvaraintofkeplerproblem}
\end{center}
\end{figure}

Step 1: Displaceability on the segment of the eigenline on the focus-focus value.
Due to the choice of $t_0$ we have that for $t^{-} < t< t_0$, $\tilde{h}(t)>1$ and thus $2R-h<R$. Hence we can apply Theorem \ref{l.displaceaingff} in Figure \ref{f.affinetrasnformationappliedtopolytopeinvariant} to conclude that the focus-focus fiber and every fiber of the form $F_t^{-1}(0,b)$, where $(0,b)\in F_t(M)$ with $b>1-2t$, is displaceable. 

 Step 2: Displaceability outside the segment of the eigenline for values in $J^{-1}(0)$.
Notice that outside the segment of the eigenline the semitoric system is a toric fibration. Therefore, use Lemma \ref{l.probes} to displace every fiber in the diagonal of $A\circ T_{0,h-2R}(f_{\vec{\epsilon}}(F_t(M))$ that does not lie over the red segment in Figure \ref{f.affinetrasnformationappliedtopolytopeinvariant} and is not the fiber over $(R,R)$.
\begin{figure}
\begin{center}
    \includegraphics[width=0.3\textwidth]{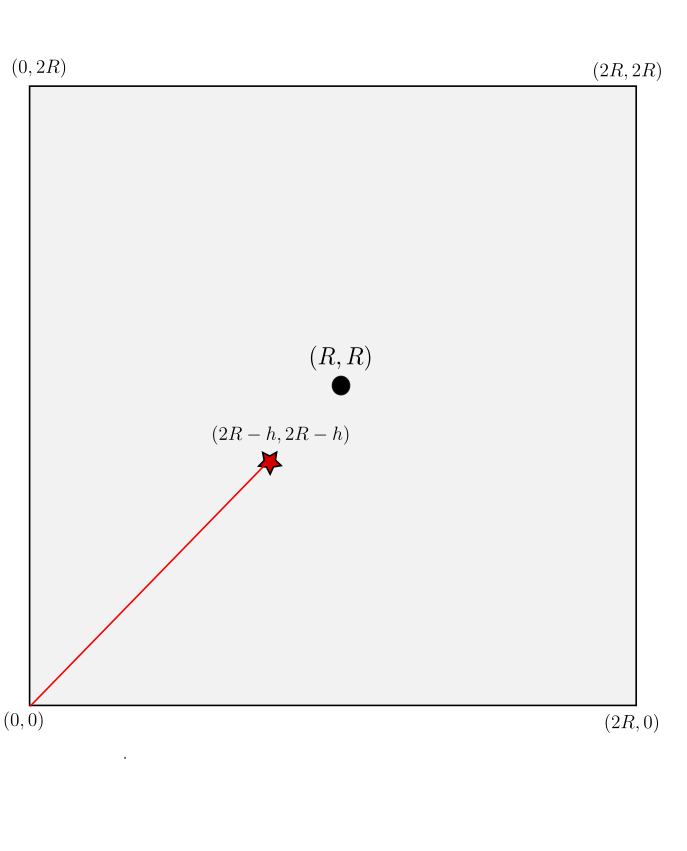}
    \caption{Image of the nodal trade of the polytope invariant by $A\circ T_{0,h-2R}$. The red line represents the image of the cut along the eigenline associated with the focus-focus fiber. The fiber corresponding to the black dot $(R,R)$ is nondisplaceable.}
    \label{f.affinetrasnformationappliedtopolytopeinvariant}
\end{center}
\end{figure}
We cannot displace the fiber over $(R,R)$ since for every probe passing trough it, the fiber is halfway the distance of the probe. 

The above results combined with Proposition \ref{p.general} show that all fibers except the fiber over $(R,R)$ in Figure \ref{f.affinetrasnformationappliedtopolytopeinvariant} are displaceable. Theorem \ref{t.existsnondisplaceable} requires the existence of at least one nondisplaceable fiber. Therefore the fiber over $(R,R)$ in Figure \ref{f.affinetrasnformationappliedtopolytopeinvariant} is nondisplaceable. In particular it is a stem, and nondisplaceable by symplectomorphisms. 
\end{proof}
We are left with understanding the cases $t\geq t_0$. In Theorem \ref{T.Keplert0} we will obtain a complete description of the (non)displaceability properties of the Kepler problem for $t=t_0$: there exists only one nondisplaceable fiber, and it is given by the focus-focus fiber, in particular it is a stem and hence nondisplaceable by symplectomorphisms. For $t>t_0$ we will see that there exist an infinite amount of nondisplaceable fibers, noticing that these fibers are equivalent to fibers already studied in the literature, see Auroux \cite{auroux2007mirror}. In particular the focus-focus fiber is nondisplaceable. This is done in Section \ref{s.anyamountnondisplaceable}.
 
We start with the following displaceability result:
\begin{proposition}
\label{p.displaceablefibersbelowfocusfocusvalue}
Consider the Kepler problem for $t_0\leq t<1$. Let $(c_0,c_0)$ denote the image of the focus-focus value by $A\circ T_{0,h-2R}\circ f_{\vec{\epsilon}}$, where $A$ and $T_{0,h-2R}$ are the maps defined in \eqref{eq.mapsfordelzantcorner} and $f_{\vec{\epsilon}}$ is the  straightening homeomorphism associated with
    the representative of the polytope invariant given in Figure \ref{f.polytopeinvaraintofkeplerproblem}. The fibers $F_t^{-1}(f_{\vec{\epsilon}}^{-1}((A\circ T_{0,h-2R})^{-1}(c,c)))$, with $c>c_0$ are \textbf{displaceable}. Therefore, since $A$ is orientation reversing, the fibers below the focus-focus value on the line $x=0$ in $F_t(M)$ are displaceable.
\end{proposition}
\begin{proof}
The proof is analogous to Step $2$ in the proof of Theorem \ref{pr.keplerOneNondispl}, it is an application of the method of probes.
\end{proof}

In order to obtain further displaceability results we need to look at another representative of the polytope invariant for the Kepler problem. We use the one given by Figure \ref{f.2ndPIkepler}, for more details see Alonso \& Dullin \& Hohloch \cite{alonso2019symplectic}. 
\begin{figure}
\begin{center}
    \includegraphics[width=0.4\textwidth]{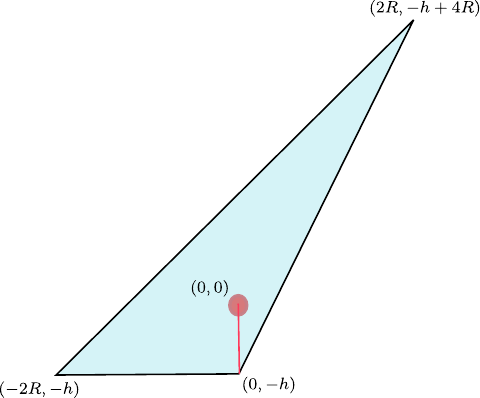}
    \caption{Representative of the polytope invariant with $\vec{\epsilon}=-1$ for the Kepler problem for $\frac{1}{5}<t<1$. The focus-focus values is represented in red. The cut along the eigenline is sketched in red.}
    \label{f.2ndPIkepler}
\end{center}
\end{figure}

\begin{proposition}
\label{p.displaceableaboveffvalue}
Consider the Kepler problem for $t_0\leq t<1$. Let $f_{\vec{\epsilon}}$ be the straightening homeomorphism associated with the polytope invariant given by Figure \ref{f.2ndPIkepler}. Then the fibers $F_t^{-1}\circ f_{\vec{\epsilon}}^{-1}(0,y)$ with $(0,y)\in f_{\vec{\epsilon}}(F_t(M))$ and $y>\frac{4}{3}R-h$ of the Kepler problem are \textbf{displaceable}. 
\end{proposition}
\begin{proof}
The idea of the proof is to use Lemma \ref{l.probes} on the polytope invariant given by Figure \ref{f.2ndPIkepler}, making sure that the probe does not intersect the segment of the eigenline associated with the focus-focus fiber. We can use Lemma \ref{l.probes} since outside the segment of the eigenline the semitoric system is toric. For this purpose we note that due the choice of $t$ we have $h\leq R$. 

Let $(0,y)\in f_{\vec{\epsilon}}(F_t(M))$ and $y>\frac{4}{3}R-h$. Notice that the vector $v:=(1,0)$ is integrally transverse to the facet given by $y=x+(2R-h)$. Therefore we can construct the probe $(x+k,y)$ where $y=x+(2R-h)>0$. The affine distance of this probe is $2R-\frac{(y+h)}{2}$ since it will end at the facet given by $y=2x-h$. Due to Lemma \ref{l.probes} the $(0,y)\in f_{\vec{\epsilon}}(F_t(M))$ such that $y>\frac{4}{3}R-h$ give displaceable fibers, since they lie less than halfway along the probe. 
\end{proof}

We can actually improve the results of Proposition \ref{p.displaceableaboveffvalue} by using nodal slides:
\begin{proposition}
\label{p.displacingEP}
Consider the Kepler problem for $t_0\leq t<1$. Let $f_{\vec{\epsilon}}$ be the straightening homeomorphism associated with the polytope invariant given by Figure \ref{f.2ndPIkepler}. Then the fibers $F_t^{-1}\circ f_{\vec{\epsilon}}^{-1}(0,y)$ with $(0,y)\in f_{\vec{\epsilon}}(F_t(M))$ and $y>R-h$ of the Kepler problem are \textbf{displaceable}. 
\end{proposition}
\begin{proof}
The fibers we consider here are (fiberwise) equivalent to the fibers that sit above the segment of the eigenline $(x,x)$ in Figure \ref{f.affinetrasnformationappliedtopolytopeinvariant} with $x<2R-h$. In the literature these are known as Chekanov-type tori, see Auroux \cite{auroux2007mirror}. 

By applying a nodal slide, see Proposition \ref{p.nodalslide}, we can make the segment of the eigenline of Figure \ref{f.2ndPIkepler} as small as possible while preserving the fiber we are considering.  This corresponds to considering the fiber $(0,y)$ in the polytope invariant, given by Figure \ref{f.2ndPIkepler}, of the system for a  $\tilde{t}>t$. Therefore the fiber over $(0,y)$ is displaceable by the probe that starts in the facet $y=x+(2R-h)$ with direction $(0,-1)$.
Hence we obtain our desired result.
%Notice that when $t$ varies the effect this has on the polytope invariants given by figure \ref{f.2ndPIkepler} is that they are related by a nodal slide, see Section \ref{s.nodaltrade}. Therefore, due to Proposition \ref{p.nodalslide}, there exists a symplectomorphism sending the fiber over $(x,x)$ to the fiber over $(x,x)$ in the polytope invariant given by figure \ref{f.2ndPIkepler} of a $t'>t$. We can make the eigenline as small as possible while preserving the fiber. Therefore the fiber over $(x,x)$ in the polytope invariant obtained by sliding the node is displaceable by the probe that starts in the facet $y=x+(2R-h)$ with direction $(0,-1)$, hence we obtain our desired result.
\end{proof}

We can now combine all of the previous results together to obtain a complete description of the (non)displaceability of the fibers of the Kepler problem $(M,\omega,F_{t_0})$.

\begin{theorem}
\label{T.Keplert0}
Consider the Kepler problem $(M,\omega,F_{t_0})$. The unique nondisplaceable fiber is the focus-focus fiber. All other fibers are displaceable. The focus-focus fiber is a stem and nondisplaceable by symplectomorphisms. 
\end{theorem}
\begin{proof}
The idea of the proof is to combine the previous results to obtain that every fiber other than the focus-focus fiber is displaceable. Indeed now in detail:
\begin{itemize}
\item By Proposition \ref{p.general} every fiber of the form $F_{t_0}^{-1}(a,b)$ with $a\neq 0$ is displaceable.
\item By Proposition \ref{p.displaceablefibersbelowfocusfocusvalue} every fiber of $F_{t_0}$ below the focus-focus value on the line $x=0$ is displaceable. 
\item Recall that due to the definition of $t_0$, we have $R=h$. By Proposition \ref{p.displacingEP} every fiber of $F_{t_0}$ above the focus-focus value on the line $x=0$ is displaceable.
\end{itemize}
Therefore the focus-focus fiber is a stem, and hence nondisplaceable by symplectomorphisms, see Theorem \ref{t.stem}.
\end{proof}

\begin{remark}
Since the fiber for the Kepler problem $(M,\omega,F_{t_0})$ is actually monotone, one could use the methods of Hong \& Kim \& Lau \cite{hong2023immersed} or Rizzel \& Ekholm \& Tonkonog \cite{rizell2018refined} to obtain the potential function for the focus-focus fiber, and then obtain a nondisplaceability result by computing the critical points of the potential function. However the result proven in Theorem \ref{T.Keplert0} is stronger since we prove that the focus-focus fiber is a stem.  
\end{remark}

%%%%%%%%%%%%%%%%%%%%%%%%%%%%%%%%%%%%%%%%%%%%%%%%%%%%
%%%%%%%%%% new subsection  %%%%%%%%%%%%%%%%%%%%%%%%%%%

\subsubsection{The case $t>t_0$ for the Kepler problem.}
\label{s.anyamountnondisplaceable}
 In this subsection we show that if $t>t_0$ the number of nondisplaceable fibers for the Kepler problem is \textbf{infinite}. In particular the focus-focus fiber of the systems is \textbf{nondisplaceable}.
First we start with a weaker result:
\begin{proposition}
\label{prop:KeplerNonDispl}
    Let $N\in \mathbb{N}$ and $R>0$. Then, for $1-\frac{1}{2^{N+2}}<t<1$, the Kepler problem $(M,\omega,F_t=(L,H_t))$ has at least $2^{N-1}+1$ nondisplaceable fibers.
\end{proposition}
To prove this result we will make use of the notions of symplectic quasi-states, superheavy, heavy, and pseudoheavy fibers summarized in Section \ref{s.states}. First note:
\begin{lemma}
\label{l.rescallingofquasistates}
    Let $(M,\omega)$ be a closed symplectic manifold and $\zeta$ a partial symplectic quasi-state in $(M,\omega)$. Let $R\neq 0$. Then $\zeta$ is a partial symplectic quasi-state for $(M,R \omega)$.
\end{lemma}
\begin{proof}
    This follows from the definition of a partial symplectic quasi-state and the fact that, if $X$ is a vector field satisfying $R \omega(X,\cdot) = dH$ for some $H\in C^{\infty}(M, \R)$, then $\omega(X,\cdot) = d\left(\frac{H}{R}\right) $, i.e., $X$ is the Hamiltonian vector field of $\frac{H}{R}$ w.r.t.\ $\omega$.
\end{proof}

Now consider the symplectic manifold $(\mathbb{S}^2\times \mathbb{S}^2,\omega_{\mathbb{S}^2}\oplus \omega_{\mathbb{S}^2})$ with Cartesian coordinates induced from $\R^3 \times \R^3$, and the Hamiltonians $G:=(J,H):\mathbb{S}^2\times \mathbb{S}^2\rightarrow \mathbb{R}^2$, where
\begin{equation*}
    \begin{cases}
    L(x_1,y_1,z_1,x_2,y_2,z_2):=z_1+z_2,\\
    H(x_1,y_1,z_1,x_2,y_2,z_2):= H_{1}= x_1x_2+y_1y_2+z_1z_2.
    \end{cases}
\end{equation*}
Then $(\mathbb{S}^2\times \mathbb{S}^2,\omega_{\mathbb{S}^2}\oplus \omega_{\mathbb{S}^2},G)$ is the Kepler problem for $R=1,t=1$ with a change of sign in the symplectic form.	
Set $S_c:= L^{-1}(0) \cap H^{-1}(c)= G^{-1}(0,c)$.
Entov \& Polterovich \cite{entov2009rigid} constructed a partial symplectic quasi-state $\zeta_{-1}$ such that $S_{-1}$ is $\zeta_{-1}$-superheavy. Furthermore, according to Fukaya \& Oh \& Ohta \& Ono \cite{fukaya2011spectral}, for every $c\in\ ]-1,-\frac{1}{2}]$, there exists a partial symplectic quasi-state $\zeta_{c}$ such that $S_{c}$ is $\zeta_{c}$-superheavy.

By Lemma \ref{l.rescallingofquasistates}, if  $R\neq 0$, then all $\zeta_{c}$ with $c\in \left[-1,-\frac{1}{2}\right]$ are partial symplectic quasi-states for $(\mathbb{S}^2\times \mathbb{S}^2, R(\omega_{\mathbb{S}^2}\oplus \omega_{\mathbb{S}^2}))$.

\begin{proof}[Proof of Proposition \ref{prop:KeplerNonDispl}]
Consider such an $S_c$ which is superheavy with respect to $\zeta_c$ as introduced above.
Let $N \in \N$ and set
\begin{equation*}
A:=\left\{ \left. -\frac{1}{2}-\frac{k}{2^N}\right| k \in \{0, \dots, 2^{N-1}\} \right\}.
\end{equation*}
By Theorem \ref{t.pseudo}, there exists, for all $c\in A$, a pseudoheavy fiber $F_t^{-1}(p_c)$ for some $p_c\in F_t(\mathbb{S}^2\times \mathbb{S}^2)$ for the integrable systems $(\mathbb{S}^2\times \mathbb{S}^2, R (\omega_{\mathbb{S}^2}\otimes \omega_{\mathbb{S}^2}),F_t)$ with $0\leq t\leq 1$ and $R\neq 0$.
In particular, by Theorem \ref{t.properties},
a pseudoheavy fiber has to intersect the corresponding superheavy fiber, i.e., 
$F^{-1}_t(p_c)\cap S_c\neq \emptyset$. Furthermore the set $F^{-1}_t(p_c)$ is nondisplaceable. Now we have to ensure that the sets $F^{-1}_t(p_c)$ are disjoint for different values of $c$. First notice that
\begin{align*}
    &\qquad  |H_t(x_1,y_1,z_1,x_2,y_2,z_2)-H(x_1,y_1,z_1,x_2,y_2,z_2)|\\
    & \quad = |(1-t)z_1+t(x_1x_2+y_1y_2+z_1z_2)-(x_1x_2+y_1y_2+z_1z_2)|\\
    & \quad = |(1-t)z_1-(1-t)(x_1x_2+y_1y_2+z_1z_2)|\\
    &\quad \leq (1-t)4.
\end{align*}
Therefore,
\begin{equation*}
    H_t(S_c)\subset [c-(1-t)4,\ c+(1-t)4].
\end{equation*}
Hence, if $t>1-\frac{1}{2^{N+2}}$, for $c\in A$ the sets $[c-(1-t)4,\ c+(1-t)4]$ do not intersect each other. Therefore the nondisplaceable fibers $F_t^{-1}(p_c)$ are disjoint. Hence we obtain the desired result.
\end{proof}
 
Proposition \ref{prop:KeplerNonDispl} tell us that the number of nondisplaceable fibers is increasing as $t\rightarrow 1$. In Lemma \ref{l.chekanov} we will improve this result by noticing that when $t>t_0$ the number of nondisplaceable fibers is actually \textbf{uncountable} and in particular the focus-focus fiber of the system is \textbf{nondisplaceable}.
\begin{lemma}
\label{l.chekanov}
Let $1>t>t_0$ and $f_{\vec{\epsilon}}$ be the straightening homeomorphism associated with the polytope invariant given by Figure \ref{f.2ndPIkepler}. Then the Lagrangian torus fibers $F_t^{-1}\circ f_{\vec{\epsilon}}^{-1}(0,y)$ of $(M,\omega,F_t)$ with $R-h\geq y>0$ are nondisplaceable. Furthermore, the focus-focus fiber of $(M,\omega,F_t)$ is nondisplaceable. 
\end{lemma}
\begin{proof}
The Lagrangian torus fibers under consideration are Chekanov-type tori, see \cite{auroux2007mirror}. Furthermore due to the condition on $y$ the nondisplaceability follows from Tonkonog \& Vianna \cite[Section $3$]{tonkonog2015low} or Auroux \cite{auroux2007mirror}, where the potential function of these tori is computed. For this choice of $y$ the potential function has a critical point, introducing a so called bulk deformation if necessary. For more details see for example Fukaya \& Oh \& Ohta \& Ono \cite{fukaya2012toric}. Therefore these tori are nondisplaceable. The focus-focus fiber is nondisplaceable since the above mentioned tori are nondisplaceable and since displaceability is an open property.
%\todo[]{\color{red} Kill the discussion below?}
%Alternatively consider the symplectic toric orbifold $(\tilde{M},\tilde{\omega})$ obtained when we consider the polytope $\tilde{\Delta}$ of figure \ref{f.2ndPIkepler} without the eigenline. From now on fix a value for $y$ and $t$, the argument works for each choice. Furthermore, notice that we have a symplectic embedding of a neighborhood $U$ of our tori into $(\tilde{M},\tilde{\omega})$ using action-angle coordinates. Using the properties of a flow and partitions of unity we obtain that if our tori were displaceable in $(M,\omega)$ then it would be displaceable in $(U,\omega)$ and hence in $(\tilde{M},\tilde{\omega})$. This gives us a contradiction since these tori are nondisplaceable in $(\tilde{M},\tilde{\omega})$. The rest of the proof is devoted to proving this non-displaceability result. We can think of $\tilde{\Delta}$ as $\tilde{\Delta}=\{(x,y)\in \mathbb{R}^2:l_i(x)\geq 0, i=1,2,3\}$ with
%\begin{align*}
%& l_1(x,y)=y\\
%& l_2(x,y)=y-2x\\
%& l_3(x,y)=-y+x+2R
%\end{align*}
%and our tori correspond to the tori on the line $x=0$ with $y<R$. 
%Therefore the potential for a trivial bulk deformation takes the form:
%\begin{equation}
%W_{(x,y),0}(\beta_1,\beta_2)=y_2*T^{y}+y_1^{-2}*y_2*T^{y-2x}+y_1*y_2^{-1}*T^{-y+x+2R}
%\end{equation}
%where $y_i=e^{\beta_i}$. If $y=R$ and $x=0$ then the function has critical points, hence we obtain our non-displaceability result due to Theorem \ref{t.potential}. We believe that the case $y<R$ is analogous but one needs to consider a deformation of the potential function. 
\end{proof}
\begin{remark}
Notice that we can use the results of Lemma \ref{l.chekanov} to show that for $t=t_0$ the focus-focus fiber of the system $(M,\omega,F_t)$ is nondisplaceable. However the result we otained in Theorem \ref{T.Keplert0} is stronger since there we actually proved that the focus-focus fiber is a stem and hence nondisplaceable by symplectomorphisms.
\end{remark}

We can combine all the previous results to obtain a complete description of the (non)displaceability properties of the Kepler problem $(M,\omega,F_t)$ for $t>t_0$:
\begin{theorem}
\label{c.t0nondisplaceable}
Let $t_0<t<1$ and $f_{\vec{\epsilon}}$ be the straightening homeomorphism associated with the polytope invariant given by Figure \ref{f.2ndPIkepler}. Then 
\begin{itemize}
\item For $(x,y)\in F_t(M)$ with $x\neq 0$ the fiber $F_t^{-1}(x,y)$ is \textbf{displaceable}.
\item For $R-h<y\leq 2R-h$ the fiber $F_t^{-1} \circ f_{\vec{\epsilon}}^{-1}(0,y)$ is \textbf{displaceable}.  
\item For $0<y\leq R-h$ the fiber $F_t^{-1} \circ f_{\vec{\epsilon}}^{-1}(0,y)$ is \textbf{nondisplaceable}.  
\item The focus-focus fiber is \textbf{nondisplaceable}.
\item For $(0,y)\in F_t(M)$ and $y<1-2t$ the fiber $F_t^{-1}(0,y)$ is \textbf{displaceable}. 
\end{itemize}
\end{theorem}
The content of Theorem \ref{c.t0nondisplaceable} is summarized in Picture \ref{f.t0nondisplaceable}.
\begin{figure}
\begin{center}
    \includegraphics[width=0.3\textwidth]{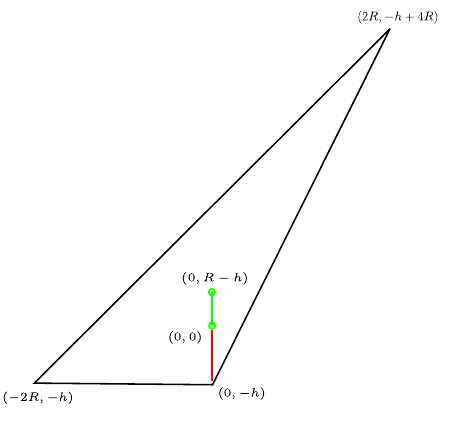}
    \caption{Representative of the polytope invariant with $\vec{\epsilon}=-1$ for the Kepler problem for $t_0<t<1$. The cut along the eigenline is sketched in red. The nondisplaceable fibers are identified in green, all the other fibers are displaceable.}
    \label{f.t0nondisplaceable}
\end{center}
\end{figure}

In the following lemma we obtain a simple result about the nonpseudoheaviness of the focus-focus fiber for the symplectic quasi-states $\zeta_{a}$ as $t\rightarrow 1$. 
\begin{lemma}
Let $a\in \ ]-1,-\frac{1}{2}]$. For $\frac{2}{3+a}<t<1$ the focus-focus fiber of the semitoric system $(M,\omega,F_t)$ is not pseudoheavy with respect to $\zeta_{a}$.
\end{lemma}
\begin{proof}
By Theorem \ref{t.properties} if the focus-focus fiber is $\zeta_a$-pseudoheavy then it intersects the superheavy fiber $S_{a}$. Therefore, there must exist a point $(x_1,y_1,z_1,x_2,y_2,z_2)\in \mathbb{S}^2\times \mathbb{S}^2$ such that 
\begin{equation}
1-2t=(1-t)z_1+ta,
\end{equation}
i.e., $z_1=\frac{1+(a-2)t}{1-t}$. If $t>\frac{2}{3+a}$ then $z_1<-1$, hence we obtain a contradiction and the desired result. 
\end{proof}

Now we summarize the results about the (non)displaceability properties of the focus-focus fiber of the systems $(M,\omega,F_t)$ in the following proposition:

\begin{proposition}Consider the Kepler problem $(M,\omega,F_t)$ for $\frac{1}{5}<t<1$. Then for
\begin{itemize}
\item $t<t_0$, the focus-focus fiber is \textbf{displaceable}.
\item $t=t_0$ the focus-focus fiber is a \textbf{stem}, hence nondisplaceable by symplectomorphisms.
\item $t>t_0$, the focus-focus fiber is \textbf{nondisplaceable}. 
\end{itemize}
\end{proposition}

\subsection{A semitoric system induced by the octagon}
\label{s.octagon}
Consider the toric symplectic manifold $(M,\omega)$ induced by the octagon $\Delta$ with vertices
\begin{equation*}
\{(1,0),(0,1),(2,0),(0,2),(1,3),(2,3),(3,1),(3,2)\}.
\end{equation*}
In De Meulenaere \& Hohloch \cite[Section $3.1$]{de2021family} it is shown that $M=L/N$, where   
\begin{equation}
\label{eq.manifoldequations}
 L:=\left\{(z_1,...,z_8)\in \mathbb{C}^8 \left| \ \begin{aligned}
&|z_1|^2+|z_5|^2=6\\
&|z_2|^2+|z_5|^2+|z_7|^2=10\\
&|z_3|^2+|z_7|^2=6\\ 
&|z_4|^2-|z_5|^2+|z_7|^2=4\\
&|z_5|^2-|z_6|^2+|z_7|^2=2\\
&|z_5|^2-|z_7|^2+|z_8|^2=4
\end{aligned}
\right.
\right\}
\end{equation}
 and $N\simeq \mathbb{T}^6$ acts on $\mathbb{C}^8$ as 
\begin{align}
\label{eq.torusaction}
   & (t,z)\mapsto  \\ & \nonumber (e^{it_1}z_1,e^{it_2}z_2,e^{it_3}z_3,e^{it_4}z_4,e^{i(t_1+t_2-t_4-t_5+t_6)}z_5,e^{it_5}z_6,e^{i(t_2+t_3+t_4-t_5-t_6)}z_7,e^{it_6}z_8).
\end{align}
Points on $(M,\omega)$ are written as equivalence classes of the form $[z]=[z_1,...,z_8]$ with $z_k=x_k+iy_k\in \mathbb{C}$ for $1\leq k\leq 8$. The momentum map of the toric system $(M,\omega)$ is given as follows:
\begin{theorem}({De Meulenaere $\&$ Hohloch, \cite[Theorem 1.1]{de2021family}})
Let $F=(J,H):M\rightarrow \mathbb{R}^2$ be given by 
\begin{equation*}
    \begin{cases}
    J([z_1,...,z_8])=\frac{1}{2}|z_1|^2,\\
    H([z_1,...,z_8])=\frac{1}{2}|z_3|^2.
    \end{cases}
\end{equation*}
Then $F$ is a momentum map of an effective Hamiltonian $2$-torus action satisfying $F(M)=\Delta$. In particular, $F$ has eight elliptic-elliptic singular points. 
\end{theorem}
An integrable system $(M,\omega,F)$ is of \textbf{toric type} if $F:M\rightarrow \mathbb{R}^2$ is proper and if there exists an effective, Hamiltonian $\mathbb{T}^2$-action on $M$ whose momentum map is of the form $f\circ F$ where $f:F(M)\rightarrow f(F(M))$ is a diffeomorphism. 
After perturbing the integral $H$ by $X:=\mathfrak{R}(Z)$, where $Z[z_1,...,z_8]:=\overline{z_2z_3z_4}z_6z_7z_8$ and $\mathfrak{R}$ stands for the real part of the function, one is able to obtain a family of semitoric systems on $(M,\omega)$. 
\begin{theorem}({De Meulenaere $\&$ Hohloch, \cite[Theorem 1.2]{de2021family}})
    Let $0<\gamma<\frac{1}{48}$. Set $F_t:=(J,H_t):=(J,(1-2t)H+t\gamma X):(M,\omega)\rightarrow \mathbb{R}^2$. Then, $(M,\omega,F_t)_{0\leq t\leq 1}$ is toric for $t=0$, of toric type for $0<t<t^{-},$ semitoric for $t^{-}<t<t^{+}$, and again of toric type for $t^{+}<t\leq 1$ where
\begin{equation*}
    0<t^{-}:=\frac{1}{2(1+24\gamma)}<\frac{1}{2}<t^{+}:=\frac{1}{2(1-24\gamma)}<1.
\end{equation*}
For all $t\in [0,1]$, the system $F_t$ has precisely eight fixed-points of which four are always elliptic-elliptic. The other four pass at $t=t^{-}$ from elliptic-elliptic via a Hamiltonian-Hopf bifurcation to focus-focus. At $t=t^{+}$, these four focus-focus points turn again back into elliptic-elliptic via a Hamiltonian-Hopf bifurcation. 
\end{theorem}
In more detail:
\begin{proposition}({De Meulenaere $\&$ Hohloch, \cite[Proposition 3.7]{de2021family}})
The focus-focus points of the semitoric systems $(M,\omega,F_t)$, for $t^{-}<t<t^{+}$ are 
\begin{center}
\begin{itemize}
\item $B:=[\sqrt{2},0,0,\sqrt{2},2,2\sqrt{2},\sqrt{6},\sqrt{6}]$;\\\item $D:=[2,\sqrt{2},0,0,\sqrt{2},\sqrt{6},\sqrt{6},2\sqrt{2}]$;\\ 
\item  $C:=[2,2\sqrt{2},\sqrt{6},\sqrt{6},\sqrt{2},0,0,\sqrt{2}]$; \\ 
\item $A:=[\sqrt{2},\sqrt{6},\sqrt{6},2\sqrt{2},2,\sqrt{2},0,0].$
\end{itemize}
\end{center}
Furthermore $A,B\in J^{-1}(1)$ and $C,D\in J^{-1}(2)$.
\end{proposition}
\begin{proposition}
    ({De Meulenaere $\&$ Hohloch, \cite[Proposition 1.3]{de2021family}}) At $t=\frac{1}{2}$, the system $F_{\frac{1}{2}}$ has precisely two focus-focus fibers, each of which contains precisely two focus-focus points so that each of these two fibers has the shape of a double pinched torus.
\end{proposition}
Our goal is to study the symplectic topology of the fibers of these semitoric systems, in particular of the focus-focus fibers. By Theorem \ref{p.lagrangiansphere}, for $t=\frac{1}{2}$ the focus-focus fibers of the semitoric system are \textbf{nondisplaceable}, due to the presence of Lagrangian spheres.
\begin{lemma}
\label{l.fffibefsymplectomorphism}
    There exist symplectomorphisms of $(M,\omega)$ that interchange the focus-focus fibers of the semitoric system $(M,\omega,F_t)$ for each $t^{-}<t<t^{+}$.
\end{lemma}
\begin{proof}
    Consider the symplectomorphisms
    \begin{align*}
        & \Psi_1(z_1,...,z_8)=(z_1,z_8,z_7,z_6,z_5,z_4,z_3,-z_2),\\
        & \Psi_2(z_1,...,z_8)=(z_5,z_4,z_3,z_2,z_1,z_8,z_7,z_6)
    \end{align*}
    of $(\mathbb{C}^8,\omega_{0})$. Recall that $(M,\omega)$ is obtained as symplectic reduction of $(\mathbb{C}^8,\omega_{0})$ on a level set $L$, see Equation \eqref{eq.manifoldequations}, by a torus action of $\mathbb{T}^6$, see Equation \eqref{eq.torusaction}. Note that $\Psi_1$ and $\Psi_2$ preserve $L$ and are invariant under the action, hence they define symplectomorphisms on $(M,\omega)$. The result follows from considering the projections of $\Psi_1,\Psi_2$,$\Psi_1\circ \Psi_2$ and $\Psi_2\circ\Psi_1$ to $(M,\omega)$. The symplectomorphisms do not depend on $t$.
\end{proof}
Since (non)displaceability is invariant under symplectomorphism it is enough to consider one of the focus-focus fibers.
%\begin{remark}
%Notice that one can view $(M,\omega)$ as an appropriate rescalling of the blow up of $\mathbb{CP}^2$ at $5$ points, Figure \ref{f.symplecticreductionoctagon}. This will be useful for the upcoming sections.
%\end{remark}
\subsubsection{Symplectic topology of the fibers of the toric system}
Before studying the symplectic topology of the fibers of the semitoric systems, we first need to understand the symplectic topology of the torus fibers given by the octagon $\Delta$.

\begin{lemma}
\label{l.probesoctagon}
Every fiber of the toric system corresponding to $\Delta$ with image of the momentum map in 
\begin{align*}
\Delta\backslash \{(1,1),(1,2),(\frac{3}{2},\frac{3}{2}),(2,1),(2,2)\}
\end{align*}
is displaceable by probes.
\end{lemma}
\begin{proof}
    Application of Lemma \ref{l.probes}: depending on the value one can use the probe with direction $(\pm 1, 0)$ or $(0,\pm 1)$, see Figure \ref{f.displaceabilitytoricoctagon}. 
\end{proof}
\begin{figure}
\begin{center}
    \includegraphics[width=0.3\textwidth]{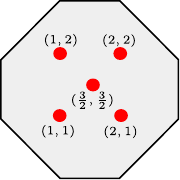}
    \caption{The fibers in the octagon that aren't displaceable by probes are highlighted at red.}
    \label{f.displaceabilitytoricoctagon}
\end{center}
\end{figure}
\noindent Now we study the (non)displaceability of the fibers of $\{(1,1),(1,2),(\frac{3}{2},\frac{3}{2}),(2,1),(2,2)\}$.
\begin{theorem}
\label{t.quasistatesoctagon}
    Let $(M,\omega)$ be the symplectic toric manifold given by $\Delta$. 
    Let $\{p_1,...,p_5\}=\{(1,1),(2,1),(\frac{3}{2},\frac{3}{2}),(1,2),(2,2)\}$.
    Then there exist $5$ different symplectic quasi-states $\zeta_1,...,\zeta_5$ for $(M,\omega)$ such that for each $p_i,\ i=1,..,5,$ the fiber $L_{p_i}$ over the point $p_i$ is \textbf{superheavy} with respect to $\zeta_i$. Therefore, the fibers $L_i$ are \textbf{nondisplaceable}.
\end{theorem}
\begin{proof}
We present the argument for the fiber over $(\frac{3}{2},\frac{3}{2})$, the other ones are analogous. See Borman \cite{borman2013quasi} and Abreu \& Macarini \cite[Section $5.4$]{abreu2013remarks} for more details on the general method.
\begin{itemize}
\item Notice that $(M,\omega)$ can be viewed as the symplectic reduction of %$\mathbb{CP}^2 \times \mathbb{C}\mathbb{P}^1 \times \mathbb{CP}^1\times \mathbb{CP}^1\times \mathbb{CP}^1$ 
$\mathbb{CP}^1\times \mathbb{CP}^1\times \mathbb{CP}^1\times \mathbb{CP}^1$
for an appropriate choice of subtorus with associated Lie algebra vectors and appropriate chosen symplectic volume of %$\mathbb{CP}^2$     
$\mathbb{CP}^1$
%, see Figure \ref{f.symplecticreductionoctagon}. 
For more details see Abreu \& Macarini \cite{abreu2013remarks}.
\item Notice that the central fiber in $(\mathbb{CP}^1,\omega_0)$ 
%and the fiber over $(\frac{1}{3},\frac{1}{3})$ in $(\mathbb{CP}^{2},\omega_{FS})$
is a stem, and in particular, superheavy for any quasi-state, see Polterovich \& Rosen \cite[Proposition $6.1.13$]{polterovich2014function}.
\item For each $\mathbb{CP}^1$ % and for \mathbb{CP}^2$
we choose quasi-states as in Borman \cite[Theorem $3.1$]{borman2013quasi} to create a product quasi-state in %$\mathbb{CP}^2\times \mathbb{CP}^1\times \mathbb{CP}^1\times \mathbb{CP}^1 \times \mathbb{CP}^1$
$\mathbb{CP}^1\times \mathbb{CP}^1\times \mathbb{CP}^1\times \mathbb{CP}^1$, see Borman \cite[Corollary $3.2$]{borman2013quasi}, such that the product of the above mentioned fibers is superheavy with respect to this quasi state. Furthermore, this quasi-state satisfies the PB-inequality. 
\item Now, apply Theorem \ref{t.reductionquasistates} to obtain the desired result. The symplectic volume of $\mathbb{CP}^1$ %and $\mathbb{CP}^2$
are chosen such that the superheavy fiber obtained as the product of superheavy fibers projects onto the fiber over the point $(\frac{3}{2},\frac{3}{2})$ in $\Delta$. 
\end{itemize}
\end{proof}
%\begin{figure}
%\begin{center}
%    \includegraphics[width=0.5\textwidth]{symplecticreductionoctagon_Layer_1.pdf}
%    \caption{$(M,\omega)$ as the symplectic reduction of $\mathbb{CP}^2\times \mathbb{CP}^1\times \mathbb{CP}^1 \times \mathbb{CP}^1\times \mathbb{CP}^1$. The dotted lines come in pairs, and each pair corresponds to a reduction by $\mathbb{CP}^1$. The triangle corresponds to $\mathbb{CP}^2$.}
%    \label{f.symplecticreductionoctagon}
%\end{center}
%\end{figure}

\subsubsection{Symplectic topology of the fibers of the semitoric systems}
Now that we understand the rigidity properties of the fibers of the toric system given by the octagon we are ready to study the semitoric case. When there are $4$ focus-focus values, for an appropriate choice of $\vec{\epsilon}$, the polytope invariant of the semitoric systems is given by the octagon, see Figure \ref{f.polytopeinvariantsemitoricoctagon}.
\begin{figure}
\begin{center}
    \includegraphics[width=0.2\textwidth]{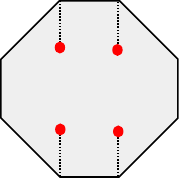}
    \caption{Representative of the polytope invariant for the semitoric system when $\vec{\epsilon}=(-1,1,-1,1)$. The dotted lines represent the eigenrays associated with the focus-focus fibers. The red circles represent the $4$ different focus-focus values.}
    \label{f.polytopeinvariantsemitoricoctagon}
\end{center}
\end{figure}
Throughout the text we use the polytope invariant of the semitoric for different choices of $\vec{\epsilon}$. We refer to Appendix \ref{s.computingpolytope} for a detailed computation of these representatives of the polytope invariant.
\begin{lemma}
\label{l.nondisplaceableexistslevelset}
Consider the semitoric system given by $(M,\omega,F_t=(J,H_t))$ for $t^{-}<t<t^{+}$. Then there exist nondisplaceable fibers $F_1,F_2,F_3$ such that $F_1\in J^{-1}(1)$, $F_2\in J^{-1}(2)$ and $F_3\in J^{-1}(\frac{3}{2})$.
\end{lemma}
\begin{proof}
We prove that $F_1$ exists and that $F_1\in J^{-1}(1)$. The proof for $F_2$ and $F_3$ is analogous. 

Recall that for the toric system given by the octagon we have two $\textbf{superheavy}$ fibers that lie inside $J^{-1}(1)$, namely $L_{(1,1)}$ and $L_{(1,2)}$ with respect to the symplectic quasi-states $\zeta_{(1,1)}$ and $\zeta_{(1,2)}$, see Theorem \ref{t.quasistatesoctagon}.

Using Theorem \ref{t.pseudo} we obtain that for the semitoric systems given by $F_t$ there exist \textbf{pseudoheavy} fibers $F^{-1}_t(y_{(1,1)})$ and $F_t^{-1}(y_{(1,2)})$, for some $y_{(1,1)},y_{(1,2)}\in \mathbb{R}^2$, with respect to $\zeta_{(1,1)}$ and $\zeta_{(1,2)}$. By Theorem \ref{t.properties} both fibers lie inside $J^{-1}(1)$, since they must intersect the corresponding \textbf{superheavy} fibers of the corresponding quasi-state. Therefore a \textbf{nondisplaceable} fiber for the semitoric system $(M,\omega,F_t)$ must exist inside $J^{-1}(1)$.
\end{proof}

%\begin{remark}
 %   To the authors knowledge there is no literature computing the height invariant of these semitoric systems, hence we will study the system for different choices of the height invariant.
%\end{remark}

Henceforth, throughout this subsection, we focus on the semitoric systems $(M,\omega,F_t)$ that have $4$ focus-focus values. The proof of the next result uses the method of probes, Lemma \ref{l.probes}, similarly as the proof of Lemma \ref{l.probesoctagon}:

\begin{lemma}
\label{l.displacingeverythingelse}
    Consider the semitoric system $(M,\omega,F_t)$, where $t^{-}<t<t^{+}$. All the fibers corresponding to points on the polytope invariant given by Figure \ref{f.polytopeinvariantsemitoricoctagon} that do not lie on the eigenrays and are not the points $\{(1,1),(2,1),(1,2),(2,2),(\frac{3}{2},\frac{3}{2})\}$ are \textbf{displaceable}.
\end{lemma}
%\begin{proof}
 %   Application of the method of probes (Lemma \ref{l.probes}).
%\end{proof}

\begin{lemma}
\label{l.octagoncentralnondisplaceable}
    Consider the semitoric system $(M,\omega,F_t)$, where $t^{-}<t<t^{+}$.
    The torus fiber corresponding to the point $(\frac{3}{2},\frac{3}{2})$ on a representative of the polytope invariant of the system $(M,\omega,F_t)$, where $t^{-}<t<t^{+}$, is nondisplaceable.
\end{lemma}
\begin{proof}
    First note that the fiber does not depend on the choice of representative of the polytope invariant since the eingenlines associated with the focus-focus values do not intersect the vertical line $x=\frac{3}{2}$. See Figure \ref{f.polytopeinvariantsemitoricoctagon} for a representative of the polytope invariant in the case $\vec{\epsilon}=(-1,1,-1,1)$.
    Applying Lemma \ref{l.probes} in a representative of the polytope invariant, we can displace every fiber in $J^{-1}(\frac{3}{2})$ that is not the fiber over the point $(\frac{3}{2},\frac{3}{2})$. Then use Lemma \ref{l.nondisplaceableexistslevelset}.

     Alternatively, we could do a nodal trade ( see Lemma \ref{l.nodaltrade}) on the representative of the polytope invariant for $\vec{\epsilon}=(-1,1,-1,1)$ and use Proposition \ref{t.quasistatesoctagon}.
\end{proof}

\begin{lemma}
    Consider the semitoric system $(M,\omega,F_t)$, where $t^{-}<t<t^{+}$. Suppose that the eigenlines associated with the focus-focus fibers in the representative of the polytope invariant given by Figure \ref{f.polytopeinvariantsemitoricoctagon} do not intersect the points $\{(1,1),(1,2),(2,1),(2,2)\}$. Then the torus fibers corresponding to the points 
    \begin{align*}
    \{(1,1),(1,2),(2,1),(2,2)\}
    \end{align*}
    are nondisplaceable.
\end{lemma}
\begin{proof}
    Apply a nodal trade (see Lemma \ref{l.nodaltrade}) not changing the fibration type over the considered values, and use Proposition \ref{t.quasistatesoctagon}.
\end{proof}

Let us try to understand the rigidity properties of the focus-focus fibers and of the fibers that lie over the eigenrays in Figure \ref{f.polytopeinvariantsemitoricoctagon}. We will consider the focus-focus fiber $\mathfrak{F}_0$ in $J^{-1}(1)$ with lower $H_t$ value. Due to Lemma \ref{l.fffibefsymplectomorphism} it is enough to understand the (non)displaceability properties of $\mathfrak{F}_0$ in order to understand the (non)displaceability properties of all focus-focus fibers of the system. Let $h$ be the height invariant associated with the focus-focus fiber $\mathfrak{F}_0$. We will focus on the cases $0<h\leq \frac{3}{2}$, the case $h>\frac{3}{2}$ is analogous. 

\subsubsection{Case $0<h<1$}
In this section we show that for $h<1$ the focus-focus fiber is \textbf{displaceable}. To this aim consider the polytope invariant which has a representative given by Figure \ref{f.polytopeinvariantdisplacebigh}.
\begin{figure}
\begin{center}
    \includegraphics[]{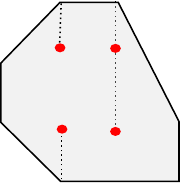}
    \caption{Representative of the polytope invariant for $\vec{\epsilon}=(-1,1,1,1)$. The dotted lines represent the eigenrays associated with the focus-focus values. The red dots represent the focus-focus values.}
    \label{f.polytopeinvariantdisplacebigh}
\end{center}
\end{figure}
After applying a suitable integral affine transformation, we obtain the polytope in Figure \ref{f.polytopetoapplynodalslide}.
\begin{figure}
\begin{center}
    \includegraphics[]{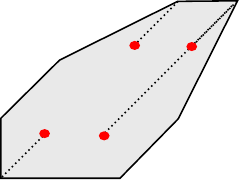}
    \caption{Image after applying a suitable integral affine transformation to a representative for the polytope invariant with $\vec{\epsilon}=(-1,1,1,1)$ and $h<1$. The dotted lines represent the eigenrays associated with the focus-focus values. The red dots represent the focus-focus values.}
    \label{f.polytopetoapplynodalslide}
\end{center}
\end{figure}

\begin{lemma}
If $h<1$ then the focus-focus fiber $\mathfrak{F}_0$ of the semitoric system is \textbf{displaceable}. Furthermore, every fiber in the eigenrays of the focus-focus fibers of Figure \ref{f.polytopeinvariantdisplacebigh} is \textbf{displaceable}.
\begin{proof}
Apply a big enough nodal slide on the eigenrays of the focus-focus fibers of Figure \ref{f.polytopetoapplynodalslide} that are not $\mathfrak{F}_0$ so that Theorem \ref{l.displaceaingff} ensures that the focus-focus fiber is displaceable. Notice that the way we displace the focus-focus fiber $\mathfrak{F}_0$ also displaces the fibers that lie on the eigenray associated with the focus-focus fiber $\mathfrak{F}_0$ in Figure \ref{f.polytopeinvariantdisplacebigh}. An analogous argument applies to the fibers on the eigenrays of the other focus-focus fibers. 

Alternatively, to displace these torus fibers, one could also use Lemma \ref{l.probes}, on the representative of the polytope invariant for $\vec{\epsilon}=(1,1,1,1)$ and $\vec{\epsilon}=(-1,-1,-1,-1)$.
\end{proof}
\end{lemma}

We can summarize the previous discussion in the following corollary:

\begin{corollary}
Consider the semitoric system given by $(M,\omega,F_t)$ and $t^{-}<t<t^{+}$ with height invariant $0<h<1$. Then there exist $5$ \textbf{nondisplaceable} fibers. These are the fibers over the points $\{(1,1),(1,2),(\frac{3}{2},\frac{3}{2}),(2,1),(2,2)\}$ in the polytope invariant given by Figure \ref{f.polytopeinvariantsemitoricoctagon}. All other fibers are \textbf{displaceable}. In particular the focus-focus fibers are \textbf{displaceable}. 
\end{corollary}

\subsubsection{Case $h\geq 1$}
Let us consider the height invariant greater than or equal to $1$. Figure \ref{f.case3} shows a representative of the polytope invariant in the case $1<h<\frac{3}{2}$. Notice that in the case $h=\frac{3}{2}$ we have $2$ focus-focus values instead of $4$.
\begin{figure}
\begin{center}
    \includegraphics[width=0.3\textwidth]{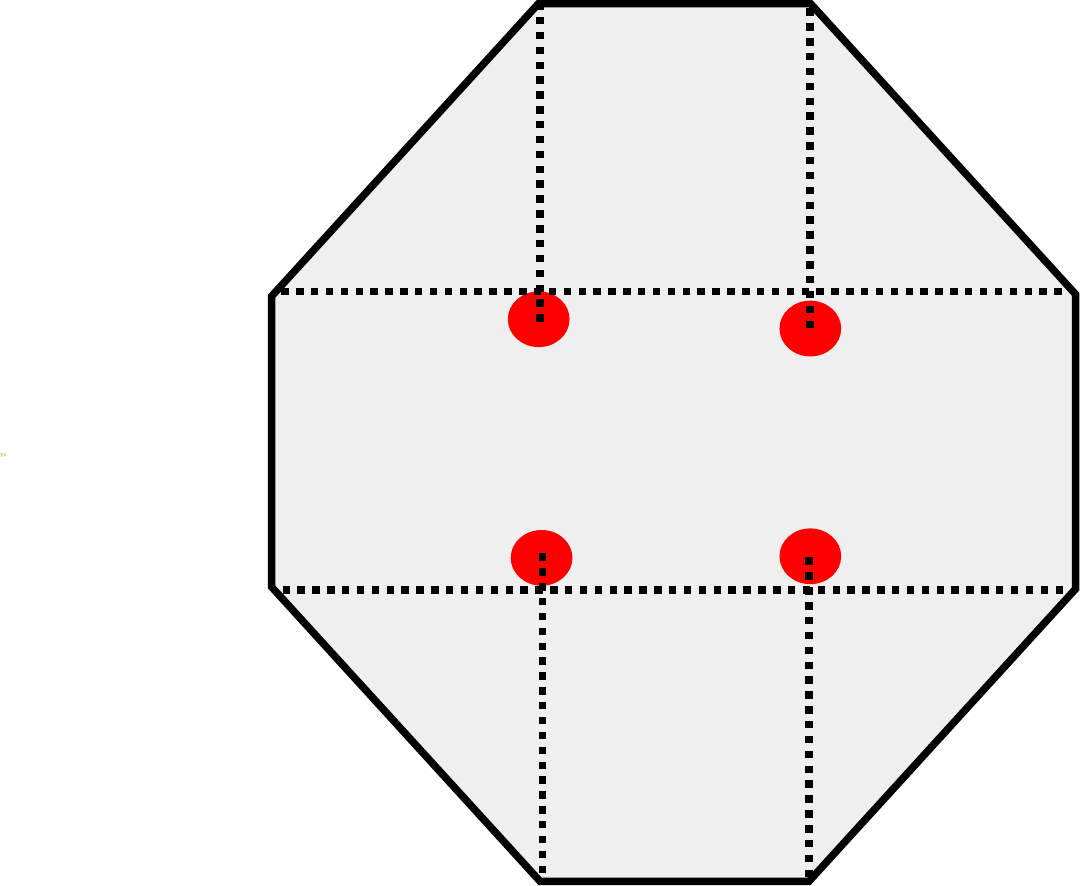}
    \caption{Polytope invariant for $\vec{\epsilon}=(-1,1,-1,1)$ with height invariant $h\geq 1$. The dotted horizontal lines indicate how big the eigenrays must be. The dotted vertical lines illustrate the eigenrays associated with the focus-focus values. The red dots represent the focus-focus values.}
    \label{f.case3}
\end{center}
\end{figure}

\subsubsection{The case $h=1$}
We show that the focus-focus fibers are \textbf{nondisplaceable}. Recall that by Lemma \ref{l.nondisplaceableexistslevelset} a \textbf{nondisplaceable} fiber exists in the level set $J^{-1}(1)$. 

Using polytopes for different choices of $\vec{\epsilon}$ and Lemma \ref{l.probes} we will prove that every fiber of $F_t$ lying inside $J^{-1}(1)$ that is not one of the focus-focus fibers must be \textbf{displaceable}. 
\begin{lemma}
\label{l.allotherlevelsetdisplaceable}
Consider the system $(M,\omega,F_t)$, where $t$ is such that $h=1$. Then every fiber in $J^{-1}(1)$ and in $J^{-1}(2)$ that is not a focus-focus fiber is \textbf{displaceable}.   
\end{lemma}
\begin{proof}
We only consider the level set $J^{-1}(1)$ since the proof for the level set $J^{-1}(2)$ is analogous. 
There are $3$ cases to consider in order to displace the desired fibers:
\begin{itemize}
\item Applying Lemma \ref{l.probes} to Figure \ref{f.firstpolyteinvariantnondisplaceable} we conclude that every fiber of the form $(f_{\vec{\epsilon}}\circ F_t)^{-1}(1,b)$ with $b<1$ is \textbf{displaceable}, where $f_{\vec{\epsilon}}$ is the straightening homeomorphism associated with the polytope invariant of Figure \ref{f.firstpolyteinvariantnondisplaceable}.
\item Applying Lemma \ref{l.probes} to Figure \ref{f.secondpolytopeinvariantnondisplaceable} we conclude that every fiber of the form $(g_{\vec{\epsilon}}\circ F_t)^{-1}(1,b)$ with $b>2$ is \textbf{displaceable}, where $g_{\vec{\epsilon}}$ is the straightening homeomorphism associated with the polytope of Figure \ref{f.secondpolytopeinvariantnondisplaceable}. 
\item Applying Lemma \ref{l.probes} to Figure \ref{f.thirdpolytopeinvariantnondisplaceable} we conclude that every fiber of the form $(h_{\vec{\epsilon}}\circ F_t)^{-1}(1,b)$ with $1<b<2$, where $h_{\vec{\epsilon}}$ is the straightening homeomorphism associated with the polytope of Figure \ref{f.thirdpolytopeinvariantnondisplaceable}, is \textbf{displaceable}.
\end{itemize}
\begin{figure}
\begin{center}
    \includegraphics[width=0.3\textwidth]{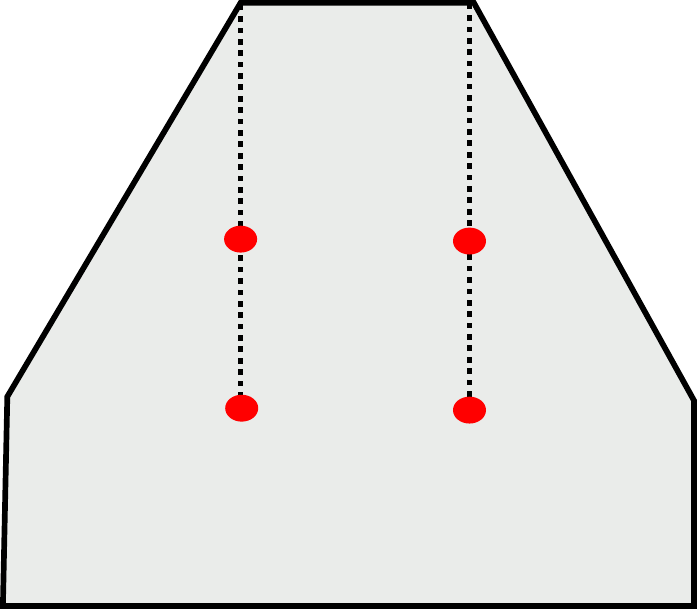}
    \caption{Representative of the polytope invariant for $\vec{\epsilon}=(1,1,1,1)$ and $h=1$. The dotted lines are the eigenrays associated with the focus-focus values. The red dots are the focus-focus values.}
    \label{f.firstpolyteinvariantnondisplaceable}
\end{center}
\end{figure}
\begin{figure}
\begin{center}
    \includegraphics[width=0.3\textwidth]{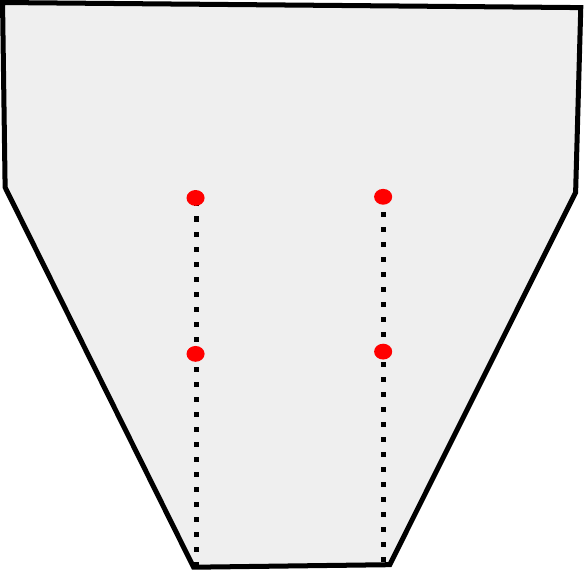}
    \caption{Representative of the polytope invariant for $\vec{\epsilon}=(-1,-1,-1,-1)$ and $h=1$. The dotted lines are the eigenrays associated with the focus-focus values. The red dots are the focus-focus values.}
    \label{f.secondpolytopeinvariantnondisplaceable}
\end{center}
\end{figure}
\begin{figure}
\begin{center}
    \includegraphics[width=0.3\textwidth]{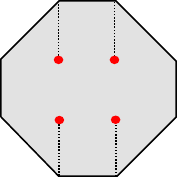}
    \caption{Representative for $\vec{\epsilon}=(-1,1,-1,1)$ and $h=1$. The dotted lines are the eigenrays associated with the focus-focus values. The red dots are the focus-focus values.}
    \label{f.thirdpolytopeinvariantnondisplaceable}
\end{center}
\end{figure}
\end{proof}

\begin{corollary}
\label{c.nondisplaceableff1}
Consider the semitoric system $(M,\omega,F_t)$, where $t$ is such that $h=1$. Then the focus-focus fibers are \textbf{nondisplaceable}. 
\end{corollary}
\begin{proof}
Combining Lemma \ref{l.allotherlevelsetdisplaceable} and Lemma \ref{l.nondisplaceableexistslevelset} we obtain that one of the focus-focus fibers in $J^{-1}(1)$ is \textbf{nondisplaceable}. Due to Lemma \ref{l.fffibefsymplectomorphism} all focus-focus fibers are \textbf{nondisplaceable}.
\end{proof}

\begin{corollary}
     Consider the semitoric system $(M,\omega,F_t)$ where $t$ is such that $h=1$. Then the focus-focus fibers and the fiber over the point $(\frac{3}{2},\frac{3}{2})$ in a representative of the polytope invariant are \textbf{nondisplaceable}. Every other fiber is displaceable. 
\end{corollary}
\begin{proof}
    Combine Lemma \ref{l.displacingeverythingelse} with Lemma \ref{l.allotherlevelsetdisplaceable}, Lemma \ref{l.octagoncentralnondisplaceable}, and Corollary \ref{c.nondisplaceableff1}.
\end{proof}

\subsubsection{The case $\frac{3}{2}\geq h>1$} This case is analogous to the previous one.

\begin{lemma}
Consider the system $(M,\omega,F_t)$, where $t$ is such that $\frac{3}{2}\geq h>1$. Then every fiber in $J^{-1}(1)$ and in $J^{-1}(2)$ that is not a focus-focus fiber is \textbf{displaceable}.   
\end{lemma}
\begin{proof}
Let $h<\frac{3}{2}$.
The proof is analogous to the proof of Lemma \ref{l.allotherlevelsetdisplaceable} with the addition that one needs to do certain nodal slides, see Proposition \ref{p.nodalslide}, on the polytope invariants of $\vec{\epsilon}=(1,1,1,1)$ and $\vec{\epsilon}=(-1,-1,-1,-1)$ in order to obtain probes with the necessary length. If $h=\frac{3}{2}$ use the polytope invariant of $\vec{\epsilon}=(1,1)$ and $\vec{\epsilon}=(-1,-1)$. 
\end{proof}
\begin{lemma}
    Consider the system $(M,\omega,F_t)$, where $t$ is such that $h=\frac{3}{2}$. Then every fiber not in $J^{-1}(1)$, not in $J^{-1}(2)$ and not the fiber over the point $(\frac{3}{2},\frac{3}{2})$ in a representative of the polytope invariant is displaceable. 
\end{lemma}
\begin{proof}
    Analogous to the proof of Lemma \ref{l.displacingeverythingelse}. Apply the method of probes to a representative of the polytope invariant. 
\end{proof}
\begin{lemma}
    Consider the system $(M,\omega,F_t)$ where $t$ is such that $h=\frac{3}{2}$. Then the focus-focus fibers are nondisplaceable.
\end{lemma}
\begin{proof}
    The focus-focus fibers have multiplicity $2$, hence they are nondisplaceable, due to the presence of Lagrangian spheres, see Theorem \ref{p.lagrangiansphere}.
\end{proof}

Combining all the previous results we obtain the following corollary:
\begin{corollary}
    Consider the semitoric system $(M,\omega,F_t)$ where $t$ is such that $1\leq h\leq \frac{3}{2}$. Then the focus-focus fibers and the fiber over the point $(\frac{3}{2},\frac{3}{2})$ in a representative of the polytope invariant are \textbf{nondisplaceable}. Every other fiber is displaceable. 
\end{corollary}

We can summarize the previous discussion in the following statement:

\begin{corollary}
Consider the semitoric system $(M,\omega,F_t)$ for $t^{-}<t<t^{+}$.
\begin{itemize}
\item If the height invariant $h$ satisfies $1\leq h< \frac{3}{2}$ then there exist $5$ \textbf{nondisplaceable} fibers. The focus-focus fibers and the fiber over the point $(\frac{3}{2},\frac{3}{2})$ in a representative of the polytope invariant. All other fibers are displaceable. 
\item If $h=\frac{3}{2}$ then there exist $3$ \textbf{nondisplaceable} fibers, the focus-focus fibers and the fiber over the point $(\frac{3}{2},\frac{3}{2})$ in a representative of the polytope invariant. All other fibers are displaceable. 
\end{itemize}
\end{corollary}

For different choices of $h$ we can summarize the (non)displaceability properties of the focus-focus fibers of the system in the following corollary:
\begin{corollary}
The focus-focus fibers have the following (non)displaceability properties with respect to different choices of the height invariant of $F_0$:
\begin{itemize}
\item If $h<1$ then the focus-focus fibers are \textbf{displaceable};
\item If $1\leq h\leq \frac{3}{2}$ then the focus-focus fibers are \textbf{nondisplaceable}.
\end{itemize}

\end{corollary}

%%%%%%%%%%%%%%%%%%%%%%%%%%%%%%%%%%%%%%%%%%%%%%%%%%%%%%%%
%%%%%%%%%%%%%%%%% new section %%%%%%%%%%%%%%%%%%%%%%%%%%%
%%%%%%%%%%%%%%%%%%%%%%%%%%%%%%%%%%%%%%%%%%%%%%%%%%%%%%%%%%%

\appendix
\section{Polytope invariant for the semitoric system induced by the octagon}
\label{s.computingpolytope}
In order to compute the representatives of the polytope invariant of the semitoric system of section \ref{s.octagon} we first need to recall some definitions and results. 

Let $(M,\omega)$ be a closed symplectic manifold of dimension $2n$. The Liouville measure on $(M,\omega)$ is the measure induced by the volume form $\frac{1}{(2\pi)^n}\frac{\omega^n}{n!}$.
The Duistermaat-Heckman measure $\mu_{J}$ for a Hamiltonian $J:M\rightarrow \mathbb{R}$ inducing an $\mathbb{S}^1$-action is defined as $\mu_{J}([a,b])=\text{vol}(J^{-1}([a,b]))$ where $\text{vol}$ is taken with respect to the Liouville measure in $M$. By Duistermaat \& Heckman \cite{duistermaat1982variation} 
\begin{equation*}
\mu_{J}=\rho_{J}(x)\frac{|dx|}{2\pi}
\end{equation*}
where the density $\rho_{J}(x)$, sometimes called the Duistermaat-Heckman function, is a continuous function, equal to the symplectic volume of the reduced orbifold $J^{-1}(x)/\mathbb{S}^1$.

Let $J$ be the momentum map of an effective Hamiltonian $\mathbb{S}^1$-action on $(M,\omega)$. For each subgroup $G\subset \mathbb{S}^1$, let $M^{G}$ be the set of points in $M$ whose stabilizer is $G$. The connected components of $M^{\mathbb{S}^1}$ are symplectic submanifolds, hence either points or surfaces, see Karshon \cite{karshon1996periodic} for more details.

\begin{lemma}
({Chaperon, \cite{chaperon1983quelques}})
Let $(M,\omega,J)$ be as above.
For each $p\in M^{\mathbb{S}^1}$ there exist neighborhoods $U\subset M$ of $p$, $U_0\subset \mathbb{C}^2$ of $(0,0)$, and a symplectomorphism $\Psi:(U,\omega)\rightarrow (U_0,\omega_0)$, where $\omega_0=\frac{i}{2}(dz_1\wedge d\overline{z}_1+dz_2\wedge d\overline{z}_2)$ making the following diagram commute
\begin{center}
\begin{tikzcd}
{(U,\omega)} \arrow[rd, "J"] \arrow[rr, "\Psi"] &            & {(U_0,\omega_0)} \arrow[ld, "J_0"'] \\
                                                & \mathbb{R} &                                    
\end{tikzcd}
\end{center}
with $J_0(z_1,z_2)=J(p)+\frac{m_1}{2}|z_1|^2+\frac{m_2}{2}|z_2|^2$.
\end{lemma}
Since the action is effective, the integers $m_1$ and $m_2$ are relatively prime. We call the integers $m_1$ and $m_2$ the \textbf{isotropy weights} at $p$. 
\begin{theorem}({\vungoc, \cite[Theorem 3.8]{vu2007moment}})
\label{t.formulapolytope}Let $(M,\omega,\Phi)$ be a semitoric system and $f_{\vec{\epsilon}}$ a straightening homeomorphism associated with a representative of the polytope invariant. Furthermore, let $x\in \mathbb{R}$ be a regular value of $J$. 
If $\alpha^{+}(x)$ (resp.\ $\alpha^{-}(x)$) denotes the slope of the top (resp. bottom) boundary of the polytope $f_{\vec{\epsilon}}\circ \Phi(M)$, then the derivative of the Duistermaat-Heckman function is 
\begin{equation*}
\rho_{J}'(x)=\alpha^{+}(x)-\alpha^{-}(x)
\end{equation*}
and is locally constant on $J(M)\setminus \{\pi_{x}(f_{\vec{\epsilon}}(\Sigma_0(\Phi))\}\in \mathbb{R}$, where $\Sigma_0(\Phi)$ is the set of critical values of $\Phi$ of maximal corank and $\pi_x$ is the projection $(x,y)\mapsto x$. If $(x,y)\in \Sigma_0(\Phi)$ then 
\begin{equation*}
\rho'_{J}(x+0)-\rho'_{J}(x-0)=-\sum_{j}k_j-e^{+}-e^{-},
\end{equation*}
where the sum runs over the set of all indices $j$ such that $\pi_{x}(c_j)=x$ and $e^{+}$ (resp. $e^{-}$) is non-zero if and only if an elliptic top vertex (resp.\ a bottom vertex) projects down onto $x$. If this occurs then 
\begin{equation*}
e^{\pm}=-\frac{1}{a^{\pm}b^{\pm}}\geq 0,
\end{equation*}
where $a^{\pm},b^{\pm}$ are the isotropy weights for the $\mathbb{S}^1$-action at the corresponding vertices. 
\end{theorem}
With Theorem \ref{t.formulapolytope}, %\ref{t.formulapolytope2}% 
we are able to compute representatives of the polytope invariant (see Section \ref{s.semitoric}) for the semitoric system given in Section \ref{s.octagon}.  We focus on the case of $4$ focus-focus values $(A,B,C,D)$. Furthermore, recall that there are $4$ elliptic-elliptic values. We label them $e_1$, $e_2$, $e_3$, $e_4$. We now compute a representative of the polytope invariant for the case $\vec{\epsilon}=(-1,1,-1,1)$, the other cases are analogous. 

%\begin{remark}
%One can see Theorem \ref{t.formulapolytope} as a special case of Theorem \ref{t.formulapolytope2} using the well known fact that the isotropy weights of the $\mathbb{S}^1$-action at isolated fixed points which are of focus-focus type are $\{+1,-1\}$, see Hohloch \& Sabatini \& Sepe \cite{hohloch2013semi}, Proposition 3.12], provided that the $\mathbb{S}^1$-action is extendable to a Hamiltonian torus action.
%\end{remark}

\begin{proposition}
Let $(M,\omega,F_t)$ for $t^{-}<t<t^{+}$ be the semitoric system defined in Section \ref{s.octagon}. Furthermore assume that $t\neq \frac{1}{2}$ so there are $4$ focus-focus values of $(M,\omega,F_t)$.
A representative of the polytope invariant for $\vec{\epsilon}=(-1,1,-1,1)$ of the semitoric system $(M,\omega,F_t)$ is the octagon given by Figure \ref{f.polytopeinvariantsemitoricoctagon}.
\end{proposition}
\begin{proof}
Consider the image of $F(M)$ under the straightening homeomorphism $f_{\vec{\epsilon}}$.
By abuse of notation we identify each $e_i$, $i=1,..,4$, with its image under $f_{\vec{\epsilon}}$.
Let $c_i$ for $i=1,...,4$ be the new vertices caused by the cut at the focus-focus critical values $(A,B,C,D)$, respectively. In order to obtain the representative of the polytope invariant we need to compute the slopes of the edges connecting these new introduced vertices with the vertices coming from the elliptic-elliptic values. Let $l_1$ be the edge connecting $e_2$ to $c_1$, $l_2$ connect $c_1$ to $c_3$, $l_3$ connect $c_3$ to $e_4$, $l_4$ connect $e_4$ to $e_3$, $l_5$ connect $c_4$ to $e_3$, $l_6$ connect $c_2$ to $c_4$, $l_7$ connect $e_1$ to $c_2$ and $l_8$ connect $e_1$ to $e_2$. Notice that $l_8$ and $l_4$ are vertical. 

First notice that due to use of an integral affine transformation that leaves a vertical line invariant we may assume that $e_1=(0,1)$ and that $l_7$ has slope $-1$. 
Let $\text{sl}$ denote the slope of an edge. Using Theorem \ref{t.formulapolytope}, %\ref{t.formulapolytope2}%
we obtain
\begin{align}
\label{eq.polytopeinvariant1}
 (\text{sl}(l_3)-\text{sl}(l_5))-(\text{sl}(l_2)-\text{sl}(l_6))& =-2\\
\label{eq.polytopeinvariant2}
 \text{sl}(l_1)-\text{sl}(l_7)& =2\\
\label{eq.polytopeinvariant3}
-(\text{sl}(l_3)-\text{sl}(l_5))& =2.
\end{align}
Since $\text{sl}(l_7)=-1$, we get $\text{sl}(l_1)=1$. 
In order to determine the slope of the other edges we use the monodromy of the system, more specifically, using item four of Theorem \ref{th.straigheningHomeo} we obtain $\text{sl}(l_2)=0$. Then using item four of Theorem \ref{th.straigheningHomeo} iteratively and equations \eqref{eq.polytopeinvariant1}, \eqref{eq.polytopeinvariant2}, \eqref{eq.polytopeinvariant3} we obtain the desired result.
\end{proof}

\printbibliography

%other way to work with bibitex:
%\bibliographystyle{alphaurl}
%\bibliography{bib}

\vspace{10mm}

\noindent 
  Sonja Hohloch $\&$ Pedro Santos\\
  \\
  University of Antwerp\\
  Department of Mathematics\\
  Middelheimlaan 1\\
  B-2020 Antwerpen, Belgium\\
  \\
  {\em E\--mail}: \texttt{sonja DOT hohloch AT uantwerpen DOT be} \\
  {\em E\--mail}: \texttt{pedro DOT santos AT uantwerpen DOT be}
  
\end{document}